\documentclass[12pt,a4paper]{amsart}
\usepackage{amsmath}
\usepackage{amscd}
\usepackage{drpack}
\usepackage[english]{babel}
\usepackage{verbatim}
\usepackage[shortlabels]{enumitem}
\usepackage{color}
\usepackage{tikz-cd}
\usepackage{mathrsfs}

\newtheoremstyle{mio}%
{}{} 
{\itshape}{} 
{\bfseries}{.}{ } 
{#1 #2\thmnote{~\mdseries(#3)}} 
\theoremstyle{mio}
\newtheorem{teor}{Theorem}[section]
\newtheorem{cor}[teor]{Corollary}
\newtheorem{prop}[teor]{Proposition}
\newtheorem{lemma}[teor]{Lemma}
\newtheorem{defin}[teor]{Definition}

\newtheoremstyle{definition2}%
{}{} 
{}{} 
{\bfseries}{.}{ } 
{#1 #2\thmnote{\mdseries~ #3}} 
\theoremstyle{definition2}
\newtheorem{ex}[teor]{Example}
\newtheorem{oss}[teor]{Remark}

\newcommand{\sigcap}{\Sigma}
\newcommand{\sgmr}{\Sigma_r}
\newcommand{\sgmi}{\Sigma_i}
\newcommand{\mmax}{\mathcal{M}}
\newcommand{\Jac}{\mathrm{Jac}}
\newcommand{\Inv}{\mathrm{Inv}}
\newcommand{\Div}{\mathrm{Div}}
\newcommand{\insid}{\mathcal{I}}
\newcommand{\deriv}{\mathcal{D}}
\DeclareMathOperator{\supp}{supp}
\DeclareMathOperator{\rk}{rk}
\newcommand{\inscrit}{\mathrm{Crit}}
\newcommand{\cons}{\mathrm{cons}}

\newcommand{\insfunct}{\mathcal{F}}
\newcommand{\insbound}{\insfunct_b}
\newcommand{\inscont}{\mathcal{C}}
\newcommand{\funcontcomp}{\mathcal{C}_c}
\newcommand{\inssubmod}{\mathbf{F}}
\newcommand{\Mod}{\mathrm{Mod}}

\title{Almost Dedekind domains without radical factorization}
\author{Dario Spirito}
\date{\today}
\address{Dipartimento di Scienze Matematiche, Informatiche e Fisiche, Universit\`a degli Studi di Udine, Udine, Italy}
\email{dario.spirito@uniud.it}

\begin{document}

\begin{abstract}
We study almost Dedekind domains with respect to the failure of ideals to have radical factorization, that is, we study how to measure how far an almost Dedekind domain is from being an SP-domain. To do so, we consider the maximal space $\mmax=\Max(R)$ of an almost Dedekind domain $R$, interpreting its (fractional) ideals as maps from $\mmax$ to $\insZ$, and looking at the continuity of these maps when $\mmax$ is endowed with the inverse topology and $\insZ$ with the discrete topology. We generalize the concept of critical ideals by introducing a well-ordered chain of closed subsets of $\mmax$ (of which the set of critical ideals is the first step) and use it to define the class of \emph{SP-scattered domains}, which includes the almost Dedekind domains such that $\mmax$ is scattered and, in particular, the almost Dedekind domains such that $\mmax$ is countable. We show that for this class of rings the group $\Inv(R)$ is free by expressing it as a direct sum of groups of continuous maps, and that, for every length function $\ell$ on $R$ and every ideal $I$ of $R$, the length of $R/I$ is equal to the length of $R/\rad(I)$.
\end{abstract}

\maketitle

\section{Introduction}

An \emph{almost Dedekind domain} is an integral domain $R$ such that, for every maximal ideal $M$ of $R$, $R_M$ is a discrete valuation ring. Almost Dedekind domains are a non-Noetherian generalization of Dedekind domains; indeed, an almost Dedekind domain is Dedekind if and only if it is Noetherian. Almost Dedekind domains appear naturally in several places of commutative algebra; for example, if the ring $\mathrm{Int}(D)$ of integer-valued polynomials on a domain $D$ in Pr\"ufer, then $D$ must be an almost Dedekind domain \cite[Proposition VI.1.5]{intD}.

Dedekind domains can also be characterized as those integral domains such that every ideal can be factorized as a product of prime ideals; for this reason, almost Dedekind domains are an interesting class in which to study weaker forms of factorizations (see for example \cite[Chapter 3]{fontana_factoring}). One of these forms is \emph{radical factorization}, that is, the possibility for an ideal to be written as a product of radical ideals. Domains where every ideal has radical factorizations are called \emph{SP-domains}, and they must be almost Dedekind \cite{vaughan-SP}, although not every almost Dedekind domain has radical factorizations. There are many equivalent characterizations of SP-domains; see \cite[Theorem 2.1]{olberding-factoring-SP} or \cite[Theorem 3.1.2]{fontana_factoring}, partly summarized in Theorem \ref{teor:caratt-SP} below.

In \cite{HK-Olb-Re}, the authors considered SP-domains with nonzero Jacobson radical from a topological point of view, and showed that, for such a ring $R$, the group $\Inv(R)$ of invertible ideals of $R$ is isomorphic (as an $\ell$-group) to the group $\inscont(\Max(R),\insZ)$ of the continuous function from the maximal space $\Max(R)$, endowed with the inverse topology, to $\insZ$, endowed with the discrete topology; in particular, this implies that $\Inv(R)$ is a free group. Subsequently, they showed that also the group $\Div(R)$ of the divisorial ideals of $R$ is a free group by interpreting it as the completion of $\Inv(R)$ as an $\ell$-group.

The purpose of this paper is to study how much of the theory in \cite{HK-Olb-Re} can be generalized to general almost Dedekind domains, and how its results must be modified in this more general context. The starting point is to associate to each fractional ideal $I$ of the almost Dedekind domain $R$ a natural map $\nu_I:\mmax\longrightarrow\insZ$, where $\mmax=\Max(R)$ is the maximal space of $R$ endowed with the inverse topology. We show that $\nu_I$ is continuous if and only if $I$ can be written as a product $J_1\cdots J_n$ of radical ideals such that every $J_i$ is the radical of a finitely generated ideal (Proposition \ref{prop:nuI-continuous}); through this result, we are first able to generalize \cite[Theorem 5.1]{HK-Olb-Re} to SP-domains with zero Jacobson radical by only considering continuous function with compact support (Corollary \ref{cor:SP-inv}), and then we define three groups measuring how much an almost Dedekind domain is far from being an SP-domain.

In Section \ref{sect:InvR}, we study a sufficient condition under which the group $\Inv(R)$ of invertible ideals is free. Generalizing the notion of \emph{critical ideals} (which appear in one of the characterizations of SP-domains), we define a descending chain of closed subsets of $\mmax$ that is analogous to the derived chain of a topological space, and we call an almost Dedekind domain \emph{SP-scattered} if this chain terminates at the empty set. For SP-scattered domains, we find an isomorphism between $\Inv(R)$ and a direct sum of groups of continuous functions (Theorem \ref{teor:SP-scattered}); in particular, this proves that for these rings $\Inv(R)$ is a free group. In the following Section \ref{sect:div}, we consider the group $\Div(R)$ of divisorial ideal of $R$; while we are not able to generalize the results of the previous section to this group, we show that for every SP-domain both $\Div(R)$ and the quotient $\Div(R)/\Inv(R)$ are free groups (Theorem \ref{teor:Div}). In the final Section \ref{sect:length} we analyze length functions over almost Dedekind domains: in particular, we show that, if $R$ is an SP-scattered domain and $I$ is an ideal, the length of $R/I$ is always equal to the length of $R/\rad(I)$.

\section{Preliminaries}

\subsection{Invertible ideals}

Let $R$ be an integral domain with quotient field $K$. A \emph{fractional ideal} of $R$ is an $R$-submodule $I$ of $K$ such that $dI\subseteq R$ for some $d\in R\setminus\{0\}$; for clarity, we call a fractional ideal contained in $R$ (i.e., an ideal of $R$) an \emph{integral ideal}. We denote by $\insfracid(R)$ the set of fractional ideals of $R$ and by $\insid(R)$ the set of integral ideals.

A fractional ideal is \emph{invertible} if there is a fractional ideal $J$ such that $IJ=R$; in this case, $J$ is unique and equal to  $(R:I):=\{x\in K\mid xI\subseteq R\}$. When $I$ is invertible, we set $I^{-1}:=(R:I)$. The set $\Inv(R)$ of invertible fractional ideals of $R$ is a group under the product of ideals.

If $T$ is an overring of $R$ (i.e., a ring between $R$ and $K$) then there is a natural map $\Psi:\Inv(R)\longrightarrow\Inv(T)$ given by the extension $I\mapsto IT$. If $R$ is a Pr\"ufer domain, then $\Psi$ is surjective (since all finitely generated ideals are invertible), and $\ker\Psi=\{I\in\Inv(R)\mid IT=T\}$.

\subsection{Almost Dedekind domains}

Let $R$ be an integral domain. We say that $R$ is an \emph{almost Dedekind domain} if $R_M$ is a discrete valuation ring for every maximal ideal $M$ of $R$. An almost Dedekind domain that is also Noetherian is necessarily a Dedekind domain, but there are examples of non-Noetherian almost Dedekind domains, the first of which was given by Nakano \cite{nakano-almded}. An almost Dedekind domain $R$ is necessarily one-dimensional and Pr\"ufer (see \cite{gilmer} for properties of Pr\"ufer domains). In particular, all its finitely generated (fractional) ideals are invertible.

An \emph{SP-domain} is an integral domain such that every proper ideals is a product of radical ideals; an SP-domain is always almost Dedekind. Among almost Dedekind domain, SP-domains enjoy several characterizations; see \cite[Theorem 2.1]{olberding-factoring-SP} and \cite[Theorem 3.1.2]{fontana_factoring} for the proof of the following equivalences.
\begin{teor}\label{teor:caratt-SP}
Let $R$ be an almost Dedekind domain. Then, the following are equivalent:
\begin{enumerate}[(i)]
\item\label{teor:caratt-SP:finprod} $R$ is an SP-domain (i.e., every proper ideal of $R$ is a finite product of radical ideals);
\item\label{teor:caratt-SPfinprodprinc} every proper principal ideal of $R$ is a (finite) product of radical ideals;
\item\label{teor:caratt-SP:prodcresc} every proper principal ideal $I$ of $R$ can be represented uniquely as $I=J_1\cdots J_k$, where $J_1\subseteq\cdots\subseteq J_k$ are radical ideals;
\item\label{teor:caratt-SP:radical} the radical of every finitely generated ideal of $R$ is finitely generated;
\item\label{teor:caratt-SP:critical} $R$ has no critical ideals (see Section \ref{sect:critical} for the definition).
\end{enumerate}
\end{teor}

We note that some of the these equivalences also holds without assuming beforehand that $R$ is an almost Dedekind domain.

\subsection{The inverse topology}
Let $R$ be an integral domain. The \emph{Zariski topology} on the spectrum $\Spec(R)$ is the topology whose open sets are those in the form $D(I):=\{P\in\Spec(R)\mid I\nsubseteq P\}$, where $I$ is an ideal; consequently, the closed sets are the $V(I):=\Spec(R)\setminus D(I)$. If $I$ is finitely generated, $D(I)$ is compact; conversely, if $D(I)$ is compact, then $D(I)=D(J)$ for some finitely generated ideal $J$.

The \emph{inverse topology} on $\Spec(R)$ is the coarsest topology such that every open and compact subset of the Zariski topology is closed. The \emph{constructible topology} is the coarsest topology such that every open and compact subset of the Zariski topology is both open and closed; the constructible topology is finer than the Zariski topology, and we denote by $\Spec(R)^\cons$ the spectrum of $R$ endowed with the constructible topology. With both the inverse and the constructible topologies, $\Spec(R)$ is a compact space; furthermore, $\Spec(R)^\cons$ is Hausdorff. See \cite[Chapter 1]{spectralspaces-libro} for properties of the inverse and constructible topology (in the topological context of spectral spaces).

Let $R$ be a ring. On the maximal spectrum $\Max(R)$ of $R$, the inverse and the constructible topology coincide \cite[Corollary 4.4.9]{spectralspaces-libro}; we denote this topological space with $\mmax$, or with $\mmax_R$ if we wan to stress the domain considered. Then, $\mmax$ is a completely regular space (in particular, Hausdorff) that is totally disconnected and has a basis of clopen subsets \cite[Corollary 4.4.9]{spectralspaces-libro}. When $R$ is one-dimensional (in particular, if $R$ is an almost Dedekind domain), $\mmax$ is compact if and only if the Jacobson radical of $R$ is nonzero.

The inverse topology is intimately connected with the study of intersection of localizations of a ring. We will use several times the following well-known lemma.
\begin{lemma}\label{lemma:closed-intersect}
Let $R$ be a one-dimensional Pr\"ufer domain and let $X\subseteq\mmax$ a nonempty closed subset of $R$. Then, the maximal ideals of $T:=\bigcap\{R_P\mid P\in X\}$ are exactly the extensions of the ideals in $X$.
\end{lemma}
\begin{proof}
Since $R$ is a one-dimensional Pr\"ufer domain, the maximal ideals of $T$ are the ideals $QT$, where $Q$ are the maximal ideals of $R$ such that $QT\neq T$ \cite[Theorem 1]{gilmer-overrings}. If $Q\in X$, then clearly $QT\neq T$. If $Q\notin X$, then $QD_P=D_P$ for all $P\in X$. Since $X$ is closed in $\mmax$, it is closed in $\Spec(R)$, endowed with the Zariski topology, and in particular it is compact; by \cite[Corollary 5]{compact-intersections},
\begin{equation*}
QT=Q\left(\bigcap_{P\in X}R_P\right)=\bigcap_{P\in X}QR_P=\bigcap_{P\in X}R_P=T.
\end{equation*}
Thus the maximal ideals of $T$ are exactly the extensions of the elements of $X$.
\end{proof}

\subsection{Isolated points}
Let $X$ be a topological space. A point $p\in X$ is \emph{isolated} if $\{p\}$ is an open set; the set of non-isolated points of $X$ is called the \emph{derived set} of $X$, and is denoted by $\deriv(X)$. More generally, if $\alpha$ is an ordinal, then $\deriv^\alpha(X)$ is defined recursively by
\begin{equation*}
\deriv^\alpha(X):=\begin{cases}
X & \text{if~}\alpha=0;\\
\deriv(\deriv^\gamma(X)) & \text{if~}\alpha=\gamma+1;\\
\bigcap_{\beta<\alpha}\deriv^\beta(X) & \text{if~}\alpha\text{~is a limit ordinal.}
\end{cases}
\end{equation*}
A topological space is \emph{scattered} if $\deriv^\alpha(X)=\emptyset$ for some ordinal $\alpha$.

A characterization closely related to the following one was obtained in \cite[Lemma 3.1]{olberding-factoring-SP}, where the topology considered on the maximal space was the Zariski topology.
\begin{prop}\label{prop:isolated-mmax}
Let $R$ be an almost Dedekind domain and $M\in\mmax$. Then, $M$ is isolated in $\mmax$ if and only if $M$ is finitely generated.
\end{prop}
\begin{proof}
If $M$ is finitely generated, then $D(M)$ is compact in the Zariski topology and thus $V(M)=\{M\}$ is an open set of $\mmax$, i.e., $M$ is isolated. Conversely, if $M$ is isolated then $\{M\}=V(M)$ is open and compact, and thus $D(M)$ is a closed set of $\Spec(R)^\cons$; hence, it is a compact open set of $\Spec(R)$ (endowed with the Zariski topology), and thus $D(M)=D(I)$ for some finitely generated ideal $I$, i.e., $M=\rad(I)$ for some finitely generated ideal $I$. Take $a\in M\setminus M^2$ (which exists since $R$ is almost Dedekind); then, $(I,a)$ is a finitely generated ideal with radical $M$ that is not contained in $M^2$. Thus $(I,a)=M$ and $M$ is finitely generated. 
\end{proof}

\subsection{Critical ideals}\label{sect:critical}
Let $R$ be an almost Dedekind domain. A maximal ideal $M$ of $R$ is said to be \emph{critical} if, for every finitely generated ideal $I\subseteq M$, there is a maximal ideal $N$ such that $I\subseteq N^2$; equivalently, $M$ is \emph{not} critical if it contains an invertible radical ideal. We denote by $\inscrit(R)$ the set of critical maximal ideals of $R$.

\begin{prop}\label{prop:crit-closed}
Let $R$ be an almost Dedekind domain. Then, $\inscrit(R)$ is a closed set of $\mmax$.
\end{prop}
\begin{proof}
Let $\Omega:=\Max(R)\setminus\bigcup\{V(I)\mid I$ is a finitely generated radical ideal$\}=\bigcap\{D(I)\mid I$ is a finitely generated radical ideal$\}$: then, by definition, $\Omega$ is closed in the inverse topology of $\Max(R)$. We claim that $\Omega=\inscrit(R)$.

Indeed, if $M\in\Omega$ and $J\subseteq M$ is finitely generated then $M\in V(J)$, and thus by definition $J$ cannot be radical; hence, there is an $N\in\Max(R)$ such that $J\subseteq N^2$, and so $M\in\inscrit(R)$. Conversely, if $M\in\inscrit(R)$, then $M$ does not contain any invertible radical ideal, and thus $M\in\Omega$. The claim is proved.
\end{proof}

\begin{prop}\label{prop:crit-deriv}
Let $R$ be an almost Dedekind domain. Then, $\inscrit(R)\subseteq\deriv(\mmax)$.
\end{prop}
\begin{proof}
If $M\notin\deriv(\mmax)$, then by Proposition \ref{prop:isolated-mmax} $M$ is finitely generated, and thus it cannot be critical.
\end{proof}

\subsection{Bounded maps and free abelian groups}
Let $X$ be a set. We denote by $\insfunct(X,\insZ)$ and by $\insbound(X,\insZ)$ the group of all function (respectively, all bounded functions) $X\longrightarrow\insZ$, with the operation being the pointwise addition. Then, $\insbound(X,\insZ)$ is a free abelian group \cite[Satz 1]{nobeling}, and thus so are all its subgroups.

Let $G\subseteq\insbound(X,\insZ)$. We say that $G$ is a \emph{Specker group} if, for every $g\in G$ and every $n\inZ$, the function $\chi_{g^{-1}(b)}$ belongs to $G$, where $\chi_A$ denotes the characteristic function of $A$. If $H\subseteq G$ are Specker groups, then $G=H\oplus S$ for some free subgroup $S$ of $G$ \cite[Satz 2]{nobeling}.

Let now $X$ be a topological space, endow $\insZ$ with the discrete topology, and let $f:X\longrightarrow\insZ$ be a function (not necessarily continuous). The \emph{zero set} of $f$ is $Z(f):=\{x\in X\mid f(x)=0\}$, while the \emph{non-zero set} (or \emph{cozero set}) is $X\setminus Z(f)$; the \emph{support} of $f$, denoted by $\supp(f)$, is the closure of its non-zero set. We say that $f$ has \emph{compact support} if $\supp(f)$ is compact, and we denote by $\insfunct_c(X,\insZ)$ the set of functions with compact support, which is a subgroup of $\insfunct(X,\insZ)$. If $X$ is compact, every function has compact support.

We denote by $\mathcal{C}(X,\insZ)$ the set of all continuous functions $X\longrightarrow\insZ$, and by $\funcontcomp(X,\insZ)$ the set of all continuous functions with compact support. Both $\inscont(X,\insZ)$ and $\funcontcomp(X,\insZ)$ are groups under pointwise addition, and thus they are subgroups of $\insfunct(X,\insZ)$.

We shall use repeatedly the following fact.
\begin{lemma}\label{lemma:cont->bound}
Let $f:X\longrightarrow\insZ$ be a continuous function. If $f$ has compact support, then $f$ is bounded.
\end{lemma}
\begin{proof}
Let $X_i:=f^{-1}(i)$. Then, the $X_i$ are open (since $\insZ$ has the discrete topology), pairwise disjoint and a cover of $\supp(f)$. Since $f$ has compact support, the cover must have a finite subcover, and thus only finitely many of them may be nonempty. It follows that $X_i=\emptyset$ when $|i|$ is large, that is, $f$ is bounded.
\end{proof}

In particular, the lemma implies that $\insfunct_c(X,\insZ)$ and $\funcontcomp(X,\insZ)$ are subgroups of $\insbound(X,\insZ)$, and thus are free.

\section{The map associated to an ideal}

Let $R$ be an almost Dedekind domain, and let $M$ be a maximal ideal. Then, $R_M$ is a discrete valuation ring, and thus it is associated to a surjective valuation map $v_M:K\setminus\{0\}\longrightarrow\insZ$. In particular, for every fractional ideal $I$ (not necessarily finitely generated) the quantity
\begin{equation*}
v_M(I):=\inf\{v_M(i)\mid i\in I\setminus\{0\}\}
\end{equation*}
is well-defined, and corresponds to the integer $n$ such that $IR_M=(MR_M)^n$ (where, if $n:=-m$ is negative, $(MR_M)^n=((MR_M)^m)^{-1}=(R_M:(MR_M)^m)$). Therefore, we can associate to each ideal $I$ a map
\begin{equation*}
\begin{aligned}
\nu_I\colon\mmax & \longrightarrow\insZ,\\
M & \longmapsto v_M(I).
\end{aligned}
\end{equation*}

\begin{prop}\label{prop:nuI}
Let $R$ be an almost Dedekind domain, and let $I,J$ be fractional ideals of $R$. Then, the following hold.
\begin{enumerate}[(a)]
\item $\nu_I\leq\nu_J$ if and only if $I\supseteq J$;
\item $\nu_I=\nu_J$ if and only if $I=J$.
\item $\nu_{IJ}=\nu_I+\nu_J$.
\item $\nu_{I+J}=\sup\{\nu_I,\nu_J\}$.
\item $\nu_{I\cap J}=\inf\{\nu_I,\nu_J\}$.
\item If $I$ is an integral ideal of $R$, then $\nu_I$ is bounded if and only if $\rad(I)^k\subseteq I$ for some $k$.
\end{enumerate}
\end{prop}
\begin{proof}
Straightforward.
\end{proof}

\begin{lemma}\label{lemma:quot-cont}
Let $R$ be an almost Dedekind domain and let $I\subseteq R$ be invertible. Then, $I=JL^{-1}$ for some invertible ideals $J,L\subseteq R$; moreover, if $\nu_I$ is continuous then we can take $J,L$ continuous.
\end{lemma}
\begin{proof}
By \cite[Theorem 25.2(c) and Remark 25.3]{gilmer}, $I(I^{-1}\cap R)=II^{-1}\cap IR=R\cap I$. Set thus $J:=I\cap R$ and $L:=I^{-1}\cap R$: then, $J$ and $L$ are invertible integral ideals of $R$ and $I=JL^{-1}$. Moreover, $\nu_J=\nu_{I\cap R}=\sup\{\nu_I,\nu_R\}=\sup\{\nu_I,0\}$ is continuous since $\nu_I$ and the zero function are continuous; hence also $\nu_L=\nu_I-\nu_J$ is continuous, as claimed.
\end{proof}

\begin{prop}\label{prop:suppnuI}
Let $R$ be an almost Dedekind domain, and let $I$ be a nonzero fractional ideal of $R$. Then, the following hold.
\begin{enumerate}[(a)]
\item\label{prop:suppnuI:integ} If $I\subseteq R$, then $Z(\nu_I)=D(I)$, and $\supp\nu_I=\Spec(D)\setminus Z(\nu_I)=V(I)$. In particular, $Z(\nu_I)$ and $\supp\nu_I$ are disjoint.
\item\label{prop:suppnuI:compact} $\supp\nu_I$ is compact.
\item\label{prop:suppnuI:invt} If $I$ is invertible, then $\supp\nu_I=\Spec(R)\setminus Z(\nu_I)$; in particular, $\supp\nu_I$ and $Z(\nu_I)$ are disjoint.
\end{enumerate}
\end{prop}
\begin{proof}
\ref{prop:suppnuI:integ} If $I\subseteq R$, then clearly $\nu_I(P)>0$ if and only if $I\subseteq P$, if and only if $P\in V(I)$. Hence, $Z(\nu_I)=\Spec(D)\setminus V(I)=D(I)$ and $\supp(I)$ is the closure of $V(I)$. Moreover, $V(I)\subseteq\mmax$ is always closed in the Zariski topology, and thus also in the constructible topology; hence $\supp\nu_I=V(I)$.

\ref{prop:suppnuI:compact} Let $d\in R$. Then, $\nu_{dR}=-\nu_{d^{-1}R}$, and thus $\supp\nu_{dR}=\supp\nu_{d^{-1}R}$. Note also that, if $I\subseteq R$, then $\supp\nu_I=V(I)$ is closed in the compact space $\Spec(D)^\cons$ and thus it is itself compact.

If now $I$ is a fractional ideal of $R$, we can find a $d\in R$, $d\neq 0$ such that $dI\subseteq R$. Then, $Z(\nu_{d^{-1}R})\cup Z(\nu_{dI})\subseteq Z(\nu_I)$, and thus
\begin{equation*}
\supp\nu_I\subseteq\supp\nu_{d^{-1}R}\cup\supp\nu_{dI}=\supp\nu_{dR}\cup\supp\nu_{dI}.
\end{equation*}
Since $d\in R$ and $dI\subseteq R$, the right hand side is compact (being the union of two compact spaces); hence, also $\supp\nu_I$ is compact, as it is the closed subset of a compact set.

\ref{prop:suppnuI:invt} By Lemma \ref{lemma:quot-cont}, $I=JL^{-1}$ with $J=I\cap R$ and $L=I^{-1}\cap R$; thus, $\nu_I=\nu_J-\nu_L$ and $\supp\nu_I\subseteq\supp\nu_J\cup\supp\nu_L$. We claim that $Z(\nu_I)=Z(\nu_J)\cap Z(\nu_L)$. Indeed, if $P$ is in the intersection then it is clearly in $Z(\nu_I)$. If $P\in Z(\nu_I)$, then $IR_P=R_P$, and thus $I\cap R\nsubseteq P$, so that $P\in Z(\nu_{I\cap R})=Z(\nu_J)$, and this forces also $P\in Z(\nu_{I^{-1}\cap R})=Z(\nu_L)$. Therefore, using \ref{prop:suppnuI:integ},
\begin{equation*}
\begin{aligned}
\supp\nu_I\subseteq & \supp\nu_J\cup\supp\nu_L=\\
= & (\Spec(D)\setminus Z(\nu_J))\cup(\Spec(D)\setminus Z(\nu_L))=\\
= & \Spec(D)\setminus (Z(\nu_J)\cap Z(\nu_L))=\\
= & \Spec(D)\setminus Z(\nu_I)\subseteq\supp\nu_I
\end{aligned}
\end{equation*}
since $\supp\nu_I$ is, by definition, the closure of $\Spec(D)\setminus Z(\nu_I)$. Hence $\supp\nu_I=\Spec(D)\setminus Z(\nu_I)$, as claimed.
\end{proof}

\begin{ex}\label{ex:noncompsupp}
For arbitrary fractional ideals $I$, $\supp\nu_I$ is not always disjoint from $Z(\nu_I)$.

Indeed, let $R$ be an almost Dedekind domain with a single noninvertible maximal ideal $M$, and suppose that there is a radical ideal $J$ that is not invertible (this surely happens if $R$ is not an SP-domain). Then, $\mmax\setminus\{M\}$ is a discrete space, and if $I$ is an ideal not contained in $M$ then $I$ is invertible and $V(I)$ is finite; in particular, $J\subseteq M$.

If $N\in\mmax\setminus\{M\}$, then $(R:J)R_N=(R_N:JR_N)$ by \cite[Lemma 5.3]{starloc} and \cite[Proposition 8.3]{jaff-derived}, and so $J(R:J)R_N=R_N$; in particular, $Z(\nu_{(R:J)})\setminus\{M\}=Z(\nu_J)\setminus\{M\}$. On the other hand, since $J$ is not invertible, $J(R:J)\subsetneq R$, and so $J(R:J)\subseteq M$. Therefore, $(\nu_J+\nu_{(R:J)})(M)\geq 1$, and since $R\subseteq(R:J)$ we have $\nu_{(R:J)}(M)=0$, so that $M\in Z(\nu_{(R:J)})$. On the other hand, $\supp\nu_{(R:J)}$ contains infinitely many points in $\mmax\setminus\{M\}$, which is a discrete set; hence, $\supp\nu_{(R:J)}$ must also contain $M$. Thus, $M\in Z(\nu_{(R:J)})\cap\supp\nu_{(R:J)}\neq\emptyset$.
\end{ex}

We now want to investigate when the map $\nu_I$ is continuous, with respect to the inverse topology. In this context, a particular importance have radical ideals.
\begin{prop}\label{prop:nuI-radical}
Let $R$ be an almost Dedekind domain, and let $I$ be an ideal of $R$.
\begin{enumerate}[(a)]
\item\label{prop:nuI-radical:rg} $I$ is radical if and only if the range of $\nu_I$ is $\{0,1\}$.
\end{enumerate}
Suppose now that $I$ is radical.
\begin{enumerate}[(a),resume]
\item\label{prop:nuI-radical:car} $\nu_I$ is equal to the characteristic function of $V(I)$.
\item\label{prop:nuI-radical:equiv} The following are equivalent:
\begin{enumerate}[(i)]
\item\label{prop:nuI-radical:equiv:cont} $\nu_I$ is continuous;
\item\label{prop:nuI-radical:equiv:clopen} $V(I)$ is clopen in $\mmax$;
\item\label{prop:nuI-radical:equiv:rad} $I=\rad(J)$ for some invertible ideal $J$.
\end{enumerate}
\end{enumerate}
\end{prop} 
\begin{proof}
\ref{prop:nuI-radical:rg} and \ref{prop:nuI-radical:car} are immediate.

In \ref{prop:nuI-radical:equiv}, \ref{prop:nuI-radical:equiv:cont} $\iff$ \ref{prop:nuI-radical:equiv:clopen} is immediate from \ref{prop:nuI-radical:car}. If $I=\rad(J)$ with $J$ finitely generated, then $V(I)=V(J)$ and $D(I)=D(J)$ are both closed in the constructible topology of $\Spec(R)$, and thus $V(I)$ and $D(I)\cap\mmax$ are both closed in $\mmax$, so that $V(I)$ is clopen; hence \ref{prop:nuI-radical:equiv:rad} $\Longrightarrow$ \ref{prop:nuI-radical:equiv:clopen}. On the other hand, if $V(I)$ is clopen, then $D(I)\cap\mmax$ is closed in $\mmax$, and thus one of $D(I)$ and $D(I)\cap\mmax$ must be closed in $\Spec(R)$, endowed with the constructible topology. However, both cases imply that $D(I)$ is compact, with respect to the Zariski topology, and thus that $I=\rad(J)$ for some finitely generated (i.e., invertible) ideal $J$. Thus \ref{prop:nuI-radical:equiv:clopen} $\Longrightarrow$ \ref{prop:nuI-radical:equiv:rad} and the claim is proved.
\end{proof}

\begin{defin}
We say that an ideal $I$ is \emph{continuous} if $\nu_I$ is continuous.
\end{defin}

\begin{lemma}\label{lemma:capX}
Let $X\subseteq\mmax$ be clopen, with respect to the inverse topology, and let $I:=\bigcap\{P\mid P\in X\}$.
\begin{enumerate}[(a)]
\item\label{lemma:capX:neq0} If $I\neq(0)$, then $I$ is a radical ideal such that $V(I)=X$.
\item\label{lemma:capX:comp} If $X$ is compact in $\mmax$ then $I\neq(0)$.
\end{enumerate} 
\end{lemma}
\begin{proof}
\ref{lemma:capX:neq0} By definition, $I$ is radical, and clearly $X\subseteq V(I)$. Furthermore, since $I\neq(0)$, $V(I)$ is a closed set of the Zariski topology contained in $\mmax$; hence, the restrictions of the Zariski, inverse and constructible topologies all agree on $V(I)$. Hence, $X$ is closed in the Zariski topology, and thus $X=V(J)$ for some radical ideal $J$ such that $I\subseteq J$. But then, $J=\bigcap\{P\mid P\in X\}=I$; hence $X=V(I)$, as claimed.

\ref{lemma:capX:comp} As $X$ is compact with respect to the inverse topology, it is also closed in $\mmax$; by Lemma \ref{lemma:closed-intersect}, the maximal ideals of $R':=\bigcap\{R_M\mid M\in X\}$ are exactly the extensions of the elements of $X$. By \cite[Corollary 4.4.17]{spectralspaces-libro}, $\Max(R')$ is closed in the constructible and thus in the Zariski topology of $\Spec(R')$; hence the Jacobson radical $\Jac(R')$ of $R'$ is nonzero. Thus, $I=\Jac(R')\cap R\neq(0)$, as claimed.
\end{proof}

\begin{prop}\label{prop:nuI-continuous}
Let $R$ be an almost Dedekind domain, and let $I$ be an integral ideal of $R$. Then, the following are equivalent:
\begin{enumerate}[(i)]
\item\label{prop:nuI-continuous:cont} $\nu_I$ is continuous;
\item\label{prop:nuI-continuous:radsubs} $I=J_1\cdots J_k$ for some continuous radical ideals $J_1\subseteq\cdots\subseteq J_k$;
\item\label{prop:nuI-continuous:radg} $I=J_1\cdots J_k$ for some continuous radical ideals $J_1,\ldots,J_k$.
\end{enumerate}
\end{prop}
\begin{proof}
\ref{prop:nuI-continuous:cont} $\Longrightarrow$ \ref{prop:nuI-continuous:radsubs} By Proposition \ref{prop:suppnuI}, $\nu_I$ has compact support, and thus it is bounded by Lemma \ref{lemma:cont->bound}, say $\nu_I(M)\leq k$ for all $M\in\mmax$.

Let $X_n:=\nu_I^{-1}([n,+\infty))$. Since $\nu_I$ is continuous, each $X_n$ is clopen, with respect to the inverse topology; moreover, $X_n\subseteq V(I)$ for all $n>1$ and $X_n$ is empty for $n>k$. By Lemma \ref{lemma:capX}, the ideals $J_n:=\bigcap\{P\mid P\in X_n\}$ are radical ideals such that $V(J_n)=X_n$, and thus $J_1\subseteq J_2\subseteq\cdots\subseteq J_k$. Moreover, by Proposition \ref{prop:nuI-radical}, the $J_n$ are continuous. We claim that $I=J_1\cdots J_k$. Indeed, let $P\in\mmax$. Then, $J_n\subseteq P$ if and only if $P\in X_n$, if and only if $\nu_I(P)\geq n$; since $J_n$ is radical, this means that $\nu_{J_n}(P)=1$ if $\nu_I(P)\geq n$ while $\nu_{J_n}(P)=0$ otherwise. Thus,
\begin{equation*}
\nu_{J_1\cdots J_k}(P)=\sum_{i=1}^k\nu_{J_i}(P)=n=\nu_I(P).
\end{equation*}
Since $P$ was arbitrary, $\nu_I=\nu_{J_1\cdots J_k}$ and so $I=J_1\cdots J_k$, as claimed.

\ref{prop:nuI-continuous:radsubs} $\Longrightarrow$ \ref{prop:nuI-continuous:radg} is obvious.

\ref{prop:nuI-continuous:radg} $\Longrightarrow$ \ref{prop:nuI-continuous:cont} Since $I=J_1\cdots J_k$, we have $\nu_I=\nu_{J_1}+\cdots+\nu_{J_k}$, and since each $\nu_{J_i}$ is continuous by hypothesis also $\nu_I$ is continuous.
\end{proof}

When $I$ is invertible, the previous proposition holds (with the same proof) if we impose that the $J_i$ are not only continuous, but also invertible. However, we can find two more useful characterizations; to do so we need some preliminaries.
\begin{prop}\label{prop:crit-VI}
Let $R$ be an almost Dedekind domain and let $I$ be an ideal of $R$ that is the radical of a finitely generated ideal. Then, $I$ is finitely generated if and only if $\inscrit(R)\cap V(I)$ is empty.
\end{prop}
\begin{proof}
If $I$ is finitely generated and $M\in V(I)$, then $M$ is not critical because $I$ itself does not satisfy the condition defining a critical ideal. Conversely, suppose that $\inscrit(R)\cap V(I)=\emptyset$: therefore, for every $M\in V(I)$ we can find a radical finitely generated ideal $J_M\subseteq M$. In the inverse topology, $\{V(J_M)\}_{M\in V(I)}$ is an open cover of the compact space $V(I)$, and thus there is a finite subcover $\{V(J_1),\ldots,V(J_k)\}$. Let $J_0$ be a finitely generated ideal such that $\rad(J_0)=I$, and let $J:=J_0+(J_1\cap\cdots\cap J_k)$: then, $J$ is radical since $J_1\cap\cdots\cap J_k$ is radical, and it is finitely generated since the $J_i$ are finitely generated and $R$ is Pr\"ufer. Moreover,  $V(J)=V(J_0)\cap(V(J_1)\cup\cdots\cup V(J_k))=V(J_0)=V(I)$ since every $M\in V(I)$ is in some $V(J_i)$. Therefore, $J$ must be equal to $I$, which thus is finitely generated.
\end{proof}

\begin{lemma}\label{lemma:radical-invt-sup}
Let $R$ be an almost Dedekind domain, and let $I\subseteq J$ be radical ideals. If $I$ is invertible and $D(J)$ is compact, with respect to the Zariski topology, then $J$ is invertible.
\end{lemma}
\begin{proof}
Since $D(J)$ is compact, there is a finitely generated ideal $J'$ such that $D(J')=D(J)$, and thus $V(J)=V(J')$. Then, $J=I+J'$ is finitely generated, and since $R$ is a Pr\"ufer domain then $J$ must also be invertible.
\end{proof}

\begin{prop}\label{prop:nuI-continuous-invt}
Let $R$ be an almost Dedekind domain, and let $I$ be an integral invertible ideal of $R$. Then, the following are equivalent:
\begin{enumerate}[(i)]
\item\label{prop:nuI-continuous-invt:cont} $\nu_I$ is continuous;
\item\label{prop:nuI-continuous-invt:radsubs} $I=J_1\cdots J_k$ for some continuous invertible radical ideals $J_1\subseteq\cdots\subseteq J_k$;
\item\label{prop:nuI-continuous-invt:radg} $I=J_1\cdots J_k$ for some continuous invertible radical ideals $J_1,\ldots,J_k$;
\item\label{prop:nuI-continuous-invt:radI} $\rad(I)$ is invertible;
\item\label{prop:nuI-continuous-invt:crit} $V(I)\cap\inscrit(R)=\emptyset$.
\end{enumerate}
\end{prop}
\begin{proof}
The equivalence of \ref{prop:nuI-continuous-invt:cont}, \ref{prop:nuI-continuous-invt:radsubs} and \ref{prop:nuI-continuous-invt:radg} follows from Proposition \ref{prop:nuI-continuous} and the fact that, when $I=J_1\cdots J_k$, the ideal $I$ is invertible if and only if each $J_i$ is invertible.


\ref{prop:nuI-continuous-invt:radsubs} $\Longrightarrow$ \ref{prop:nuI-continuous-invt:crit} By Proposition \ref{prop:crit-VI}, each $V(J_i)$ is disjoint from $\inscrit(R)$. Hence, also $V(I)=V(J_1\cdots J_k)=V(J_1)\cup\cdots\cup V(J_k)$ is disjoint from $\inscrit(R)$.

\ref{prop:nuI-continuous-invt:crit} $\Longrightarrow$ \ref{prop:nuI-continuous-invt:radI} We have $V(\rad(I))\cap\inscrit(R)=V(I)\cap\inscrit(R)=\emptyset$. By Proposition \ref{prop:crit-VI}, $\rad(I)$ is finitely generated.

\ref{prop:nuI-continuous-invt:radI} $\Longrightarrow$ \ref{prop:nuI-continuous-invt:cont} Suppose that $\nu_I$ is not continuous, and consider the spaces $X_t:=\nu_I^{-1}([t,+\infty))$. If they are all clopen, then also each $f^{-1}(t)=X_t\cap(\mmax\setminus X_{t+1})$ is open, and thus $\nu_I$ would be continuous, a contradiction. Let thus $s$ be the minimal index such that $X_s$ is not clopen in $\mmax$: then, $s>0$ since $X_0=\mmax$. Moreover, $X_1=\mmax\setminus X_0=\supp\nu_I=V(I)$ is clopen (as $I$ is invertible) and thus $s>1$.

For $n<s$, since $X_n$ is clopen the ideal $J_n:=\bigcap\{P\mid P\in X_n\}$ is a continuous radical ideal containing $J_1=\rad(I)$, and since by hypothesis $\rad(I)$ is invertible then also $J_1,\ldots,J_{s-1}$ are invertible by Lemma \ref{lemma:radical-invt-sup}. Let $I':=IJ_1^{-1}\cdots J_{s-1}^{-1}$: then, $I'$ is an invertible integral ideal of $R$ such that $\nu_{I'}^{-1}([t,+\infty))=\nu_I^{-1}([t+s-1,+\infty))=X_{t+s-1}$ for all $t\geq 0$. In particular, $V(I')=\nu_{I'}^{-1}([1,+\infty))=\nu_I^{-1}([s,+\infty))=X_s$. However, $V(I')$ is clopen since $I'$ is invertible, while $X_s$ is not by the definition of $s$. This is a contradiction, and thus $\nu_I$ must be continuous.
\end{proof}

In particular, if we apply Proposition \ref{prop:nuI-continuous-invt} to every invertible ideal, we get back some of the characterizations of SP-domains given in Theorem \ref{teor:caratt-SP}.

\section{Function spaces}
There are two ways in which the correspondence between continuous valuation functions $\nu_I$ and invertible ideals can fail: continuous radical ideals may not be invertible, and invertible ideals may not give rise to continuous function. To study how much these properties fail, we need to work in the context of groups of functions.

Since every $\nu_I$ is a function of compact support from $\mmax$ to $\insZ$ (Proposition \ref{prop:suppnuI}), we have a map
\begin{equation}\label{eq:defpsi}
\begin{aligned}
\Psi\colon\insfracid(R) & \longrightarrow\insfunct_c(\mmax,\insZ),\\
\nu & \longmapsto\nu_I,
\end{aligned}
\end{equation}
which is injective by Proposition \ref{prop:nuI}. We also denote by $\Psi$ the restriction of this map to $\Inv(R)$. In this last case, $\Inv(R)$ carries also a group structure; since $\nu_{IJ}=\nu_I+\nu_J$, we have $\Psi(IJ)=\Psi(I)+\Psi(J)$, that is, $\Psi$ is an injective group homomorphism from $\Inv(R)$ to $\mathcal{F}_c(\mmax,\insZ)$. 

In particular, $\Psi$ sends the set $\insid(R)$ of integral ideals inside the set $\mathcal{F}_c(\mmax,\insN)$, and conversely: that is, if $\Psi(I)=\nu_I\in\mathcal{F}_c(\mmax,\insN)$ then $I$ must be an integral ideal.

We now want to consider the set of continuous functions in the form $\nu_I$.
\begin{prop}\label{prop:funcontcomp}
Let $R$ be an almost Dedekind domain. Then:
\begin{enumerate}[(a)]
\item $\funcontcomp(\mmax,\insZ)$ is generated by the $\nu_I$, as $I$ ranges among the continuous radical ideals;
\item $\funcontcomp(\mmax,\insN)\subseteq\Psi(\insid(R))$.
\end{enumerate}
\end{prop}
\begin{proof}
Let $f\in\funcontcomp(\mmax,\insZ)$. Since $\insZ$ has the discrete topology, we can write $f=a_1\chi_{X_1}+\cdots+a_n\chi_{X_n}$, where each $X_i$ is a clopen subset of $R$ contained in $\supp f$ and each $a_i\inZ$; in particular, each $X_i$ (being a closed subset of a compact space) is itself compact, with respect to the inverse topology. By Lemma \ref{lemma:capX}, $J_i:=\bigcap\{P\mid P\in X_i\}$ is a radical ideal with $V(J_i)=X_i$: thus, $f=a_1\nu_{J_1}+\cdots+a_n\nu_{J_n}$ lies in the group generated by the $\nu_I$.

If now $f\in\funcontcomp(\mmax,\insN)$, then again we can write $f=a_1\chi_{X_1}+\cdots+a_n\chi_{X_n}$, with the $a_i\geq 0$: with the same construction we can write $f=\nu_I$ with $I:=J_1^{a_1}\cdots J_n^{a_n}$.
\end{proof}

An immediate corollary is the extension of \cite[Theorem 5.1]{HK-Olb-Re} to SP-domains with zero Jacobson radical.
\begin{cor}\label{cor:SP-inv}
Let $R$ be an SP-domain. Then, $\Inv(R)\simeq\funcontcomp(\mmax,\insZ)$.
\end{cor}
\begin{proof}
Since $R$ is an SP-domain, every continuous radical ideal is invertible (Proposition \ref{prop:nuI-radical}\ref{prop:nuI-radical:equiv} and Theorem \ref{teor:caratt-SP}). By Proposition \ref{prop:funcontcomp}, $\funcontcomp(\mmax,\insZ)$ is generated by the set of $\nu_I$, as $I$ ranges among the continuous radical ideals; by Proposition \ref{prop:nuI-continuous-invt}, these ideals generate $\Inv(R)$, and thus $\Psi(\Inv(R))$ generates $\funcontcomp(\mmax,\insZ)$. Since $\Psi$ is an injective group homomorphism when restricted to $\Inv(R)$, it follows that its image is the whole $\funcontcomp(\mmax,\insZ)$ and thus $\Inv(R)$ and $\funcontcomp(\mmax,\insZ)$ are isomorphic, as claimed.
\end{proof}

Note that $\funcontcomp(\mmax,\insZ)$, in general, does not lie inside $\Psi(\insfracid(R))$: indeed, suppose that $J$ is radical with $V(J)$ compact. Then, $\nu_J$ is continuous, and thus so is $-\nu_J$; however, $-\nu_J=\nu_L$ for some $L$ if and only if $JL=R$, that is, $-\nu_J\in\Psi(\insfracid(R))$ if and only if $J$ is invertible, which in general needs not to happen. Nevertheless, $\funcontcomp(\mmax,\insN)$ generates $\funcontcomp(\mmax,\insZ)$ as a group, and thus the group generated by $\Psi(\insid(R))$ contains $\funcontcomp(\mmax,\insZ)$.

To explore the relationship between $\Psi(\Inv(R))$ and $\funcontcomp(\mmax,\insZ)$, we will consider three groups:
\begin{itemize}
\item the first one is $\sigcap(R):=\Psi(\Inv(R))\cap\funcontcomp(\mmax,\insZ)$, the group of continuous invertible ideals;
\item the second one is the quotient $\displaystyle{\sgmr(R):=\frac{\funcontcomp(\mmax,\insZ)}{\sigcap(R)}}$, that measures how many continuous functions do not come from invertible ideals;
\item the third one is the quotient $\displaystyle{\sgmi(R):=\frac{\Psi(\Inv(R))}{\sigcap(R)}}$, that measures how many invertible ideals are not continuous.
\end{itemize}

We shall study the first two of these groups in the remainder of this section, and the last one in the following Section \ref{sect:InvR}. Before doing so, we reinterpret the characterizations of SP-domains through these groups.
\begin{prop}\label{prop:SP-gruppi}
Let $R$ be an almost Dedekind domain. The following are equivalent:
\begin{enumerate}[(i)]
\item\label{prop:SP-gruppi:SP} $R$ is an SP-domain;
\item\label{prop:SP-gruppi:cap0} $\funcontcomp(\mmax,\insZ)=\Psi(\Inv(R))$;
\item\label{prop:SP-gruppi:cap1} $\sigcap(R)=\funcontcomp(\mmax,\insZ)$;
\item\label{prop:SP-gruppi:cap2} $\sigcap(R)=\Psi(\Inv(R))$;
\item\label{prop:SP-gruppi:r} $\sgmr(R)=(0)$;
\item\label{prop:SP-gruppi:i} $\sgmi(R)=(0)$.
\end{enumerate}
\end{prop}
\begin{proof}
\ref{prop:SP-gruppi:cap1} $\iff$ \ref{prop:SP-gruppi:r} and \ref{prop:SP-gruppi:cap2} $\iff$ \ref{prop:SP-gruppi:i} are a direct consequence of the definitions; thus, it is enough to consider the first four conditions.

The domain $R$ is an SP-domain if and only if $\rad(I)$ is invertible for every integral invertible ideal $I$, and thus by Proposition \ref{prop:nuI-continuous-invt} if and only if $\nu_I$ is continuous for every invertible ideal $I$; i.e., $R$ is an SP-domain if and only if $\Psi(\Inv(R))\subseteq\funcontcomp(\mmax,\insZ)$, if and only if $\sigcap(R)=\Psi(\Inv(R))$. Thus \ref{prop:SP-gruppi:SP} $\iff$ \ref{prop:SP-gruppi:cap2}.

On the other hand, $\funcontcomp(\mmax,\insZ)$ is generated by the $\nu_I$, as $I$ ranges among the continuous radical ideals $I$; thus, $\funcontcomp(\mmax,\insZ)\subseteq\Psi(\Inv(R))$ (i.e., $\sigcap(R)=\funcontcomp(\mmax,\insZ)$) if and only if, for every radical ideal $I$, $\nu_I$ is continuous, if and only if every such $I$ is invertible. By Propositions \ref{prop:nuI-radical} and \ref{prop:nuI-continuous-invt}, this happens if and only if, whenever $I$ is the radical of a finitely generated ideal, then it is itself finitely generated, and thus if and only if $R$ is an SP-domain. Thus \ref{prop:SP-gruppi:SP} $\iff$ \ref{prop:SP-gruppi:cap1}.

Finally, it is clear that \ref{prop:SP-gruppi:cap0} implies \ref{prop:SP-gruppi:cap1} and \ref{prop:SP-gruppi:cap2}, and that \ref{prop:SP-gruppi:cap1} and \ref{prop:SP-gruppi:cap2} together imply \ref{prop:SP-gruppi:cap0}. Since \ref{prop:SP-gruppi:cap1} and \ref{prop:SP-gruppi:cap2} are both equivalent to \ref{prop:SP-gruppi:SP}, all conditions are equivalent.
\end{proof}

\begin{prop}\label{prop:sigcap-iso}
Let $R$ be an almost Dedekind domain. Then, the restriction map gives an isomorphism $\sigcap(R)\simeq\funcontcomp(\mmax\setminus\inscrit(R),\insZ)$; in particular, $\sigcap(R)$ is a free abelian group.
\end{prop}
\begin{proof}
Let
\begin{equation*}
\begin{aligned}
\Phi\colon\sigcap(R) & \longrightarrow\funcontcomp(\mmax\setminus\inscrit(R),\insZ),\\
\nu_I & \longmapsto \nu_I|_{\mmax\setminus\inscrit(R)};
\end{aligned}
\end{equation*}
be the restriction map. Then, $\Phi$ is well-defined: if $I$ is invertible, then $I=JL^{-1}$ for some invertible continuous ideals $J,L$ (Lemma \ref{lemma:quot-cont}) and by Proposition \ref{prop:nuI-continuous-invt} both $V(J)$ and $V(L)$ are disjoint from $\inscrit(R)$. Hence, $\supp\nu_I\subseteq V(J)\cup V(L)\subseteq\mmax\setminus\inscrit(R)$, and so $\nu_I|_{\mmax\setminus\inscrit(R)}$ has compact support. Moreover, this shows that if $\Phi(\nu_I)$ is zero then the whole $\nu_I$ must be zero, and thus $\Phi$ is injective.

Let now $f\in\funcontcomp(\mmax\setminus\inscrit(R),\insZ)$: then, $f=\sum_ia_i\chi_{X_i}$ for some compact clopen subsets $X_i$ of $\mmax\setminus\inscrit(R)$ and some $a_i\inZ$. We claim that the $X_i$ are clopen in $\mmax$: indeed, they are open since they are open subsets of the open set $\mmax\setminus\inscrit(R)$, while they are closed since they are compact subsets of the Hausdorff space $\mmax$. By Lemma \ref{lemma:capX}, the ideals $J_i:=\bigcap\{P\mid P\in X_i\}$ are nonzero and such that $V(J_i)=X_i$; moreover, they are invertible by Proposition \ref{prop:nuI-continuous-invt}, since $X_i\cap\inscrit(R)=\emptyset$. Setting $I:=\prod_iJ_i^{a_i}$ we have $f=\nu_I$. Thus, $\Phi$ is surjective and hence an isomorphism.

The last claim follows from the fact that $\funcontcomp(\mmax\setminus\inscrit(R),\insZ)$ is free, being a subgroup of the free group $\mathcal{C}(\mmax\setminus\inscrit(R),\insZ)$ \cite[Satz 1]{nobeling}.
\end{proof}

For $\sgmr(R)$ the result is similar, but involves $\inscrit(R)$ instead of its complement.
\begin{prop}\label{prop:funct-crit}
Let $R$ be an almost Dedekind domain. Then, $\sgmr(R)\simeq\funcontcomp(\inscrit(R),\insZ)$; in particular, $\sgmr(R)$ is a free abelian group.
\end{prop}
\begin{proof}
Consider the quotient map $\pi:\funcontcomp(\mmax,\insZ)\longrightarrow\sgmr(R)$, and let $f,g\in\funcontcomp(\mmax,\insZ)$. We claim that $\pi(f)=\pi(g)$ if and only if the restrictions of $f$ and $g$ to $\inscrit(R)$ are the same, and to do so it is enough to prove that $\pi(f)=0$ if and only if $f|_{\inscrit(R)}=0$.

Indeed, suppose that $f|_{\inscrit(R)}=0$. We can write $f=\sum_ic_i\chi_{Z_i}$, for some disjoint clopen and compact sets $Z_i$, and the hypothesis implies that no $Z_i$ meets $\inscrit(R)$. By Lemma \ref{lemma:capX}, we can find radical ideals $J_i$ such that $\nu_{J_i}=\chi_{Z_i}$; by Proposition \ref{prop:nuI-radical}, each $J_i$ is the radical of a finitely generated ideal, and since $Z_i\cap\inscrit(R)=\emptyset$, by Proposition \ref{prop:crit-VI} the $J_i$ themselves are finitely generated and thus invertible. Let $L:=\prod_iJ_i^{a_i}$: then, $f=\nu_L\in\Psi(\Inv(R))$, and thus $f$ is in the kernel of $\pi$, i.e., $\pi(f)=0$.

Conversely, suppose that $\pi(f)=0$. Then, $f=\nu_L$ for some $L\in\Inv(R)$; by Lemma \ref{lemma:quot-cont} $f=\nu_I-\nu_J$ for some invertible integral ideals $I,J$ such that $\nu_I,\nu_J$ are continuous. By Proposition \ref{prop:nuI-continuous-invt}, $I$ and $J$ are products of continuous invertible radical ideals; by Proposition \ref{prop:crit-VI}, the support of each of these radical ideals does not meet $\inscrit(R)$. Therefore, $\nu_I-\nu_J$ is the zero function on $\inscrit(R)$, i.e., $f|_{\inscrit(R)}=0$, as claimed.

The first part of the above equivalence implies that the quotient $\pi$ factors through $\funcontcomp(\inscrit(R),\insZ)$, i.e., that we have a chain of maps
\begin{equation*}
\funcontcomp(\mmax,\insZ)\longrightarrow\funcontcomp(\inscrit(R),\insZ)\longrightarrow\sgmr(R),
\end{equation*}
and the second part that the rightmost map is injective. It follows that $\funcontcomp(\inscrit(R),\insZ)$ and $\sgmr(R)$ are isomorphic, as claimed. The ``in particular'' part now follows from the fact that $\funcontcomp(\inscrit(R),\insZ)$ is free, as a subgroup of the free group $\insbound(\inscrit(R),\insZ)$ (see \cite[Satz 1]{specker} and Lemma \ref{lemma:cont->bound}).
\end{proof}

\begin{cor}
Let $R$ be an almost Dedekind domain. Then, $\sgmr(R)$ has finite rank if and only if $\inscrit(R)$ is finite; in this case, $|\rk\sgmr(R)|=|\inscrit(R)|$.
\end{cor}
\begin{proof}
The group $\funcontcomp(\inscrit(R),\insZ)$ has finite rank if and only if $\inscrit(R)$ is finite, and in this case $\funcontcomp(\inscrit(R),\insZ)=\bigoplus\{\insZ\mid M\in\inscrit(R)\}$ has rank $|\inscrit(R)|$. The claim follows from Proposition \ref{prop:funct-crit}.
\end{proof}

\begin{cor}
Let $R,R'$ be almost Dedekind domains. If $\inscrit(R)$ and $\inscrit(R')$ are homeomorphic (when they are endowed with the inverse topology) then $\sgmr(R)\simeq\sgmr(R')$.
\end{cor}
\begin{proof}
Immediate from Proposition \ref{prop:funct-crit}.
\end{proof}

In general, it may not be easy to individuate the set $\inscrit(R)$ of critical maximal ideals of $R$. In the next proposition, we give a weaker version of Proposition \ref{prop:funct-crit} that only depend on the topological structure of $\mmax$. 
\begin{prop}
Let $R$ be an almost Dedekind domain. Then, there is a surjective map $\funcontcomp(\deriv(\mmax),\insZ)\longrightarrow\sgmr(R)$.
\end{prop}
\begin{proof}
By Proposition \ref{prop:crit-deriv}, $\inscrit(R)\subseteq\deriv(\mmax)$. Since $\inscrit(R)$ is closed in $\mmax$, the restriction map $\funcontcomp(\deriv(\mmax),\insZ)\longrightarrow\funcontcomp(\inscrit(R),\insZ)$ is surjective; the claim now follows from Proposition \ref{prop:funct-crit}.
\end{proof}

\section{The freeness of $\Inv(R)$}\label{sect:InvR}
In this section we study the group $\sgmi(R):=\Psi(\Inv(R))/\sigcap(R)$. The first result is that this group can actually be reduced to an already known object.

\begin{prop}\label{prop:sgmi-inv}
Let $R$ be an almost Dedekind domain, and let $T:=\bigcap\{R_P\mid P\in\inscrit(R)\}$.  Then, $\sgmi(R)\simeq\Inv(T)$.
\end{prop}
\begin{proof}
Since $\inscrit(R)$ is closed, by Lemma \ref{lemma:closed-intersect} the maximal ideals of $T$ are the extensions of the critical ideals of $R$. Moreover, since $R$ is Pr\"ufer, the extension map $\Inv(R)\longrightarrow\Inv(T)$ is surjective, with kernel $K:=\{I\in\Inv(R)\mid IT=T\}=\{I\in\Inv(R)\mid\supp(\nu_I)\cap\inscrit(R)=\emptyset\}$. Therefore, the corresponding surjective map $\Psi(\Inv(R))\longrightarrow\Inv(T)$ has kernel
\begin{equation*}
\{\nu_I\mid \supp(\nu_I)\cap\inscrit(R)=\emptyset\}=\sigcap(R),
\end{equation*}
using Proposition \ref{prop:sigcap-iso}. Hence, $\Inv(T)\simeq\Psi(\Inv(R))/\sigcap(R)$, which is $\sgmi(R)$ by definition.
\end{proof}

Therefore, we have an exact sequence
\begin{equation*}
0\longrightarrow\sigcap(R)\longrightarrow\Inv(R)\longrightarrow\Inv(T)\longrightarrow 0.
\end{equation*}
The ring $T$ is itself an almost Dedekind domain, and thus we can apply the same reasoning: setting $T_2:=\bigcap\{T_P\mid P\in\inscrit(T)\}$, we have
\begin{equation*}
0\longrightarrow\sigcap(T)\longrightarrow\Inv(T)\longrightarrow\Inv(T_2)\longrightarrow 0
\end{equation*}
or more generally
\begin{equation*}
0\longrightarrow\sigcap(T_k)\longrightarrow\Inv(T_k)\longrightarrow\Inv(T_{k+1})\longrightarrow 0
\end{equation*}
where $T_{k+1}:=\bigcap\{(T_k)_P\mid P\in\inscrit(T_k)\}$ (and we set $T_0:=D$ and $T_1:=T$ for uniformity). For this reason, we introduce the following definition.
\begin{defin}
Let $R$ be an almost Dedekind domain. For every ordinal $\alpha$, define recursively the following:
\begin{itemize}
\item $\inscrit_0(R):=\mmax$;
\item $T_0:=R$;
\item if $\alpha=\gamma+1$ is a successor ordinal,
\begin{equation*}
\inscrit_\alpha(R):=\{P\in\mmax\mid PT_\gamma\in\inscrit(T_\gamma)\};
\end{equation*}
\item if $\alpha$ is a limit ordinal,
\begin{equation*}
\inscrit_\alpha(R):=\bigcap_{\gamma<\alpha}\inscrit_\gamma(R);
\end{equation*}
\item $T_\alpha:=\bigcap\{R_P\mid P\in\inscrit_\alpha(R)\}$.
\end{itemize}
We call the minimal ordinal $\alpha$ such that $\inscrit_\alpha(R)=\inscrit_{\alpha+1}(R)$ (equivalently, such that $T_\alpha=T_{\alpha+1}$) the \emph{SP-rank} of $R$. If $\inscrit_\alpha(R)=\emptyset$ (equivalently, if $T_\alpha=K$) for this $\alpha$, then we say that $R$ is \emph{SP-scattered}.
\end{defin}

Note that, if $\inscrit_\alpha(R)=\inscrit_{\alpha+1}(R)$, then, for all $\gamma>\alpha$, $\inscrit_\alpha(R)=\inscrit_\gamma(R)$ and $T_\gamma=T_\alpha$.

The following two lemmas generalize Propositions \ref{prop:crit-closed} and \ref{prop:crit-deriv}, and allow to give a sufficient condition for $R$ to be SP-scattered. In particular, Lemma \ref{lemma:incritalpha} can be seen as a variant of \cite[Lemma 6.5]{HK-Olb-Re}.
\begin{lemma}\label{lemma:incritalpha}
Let $R$ be an almost Dedekind domain, and let $\alpha$ be an ordinal. Then:
\begin{enumerate}[(a)]
\item for every $P\in\mmax_R$, $PT_\alpha\neq T_\alpha$ if and only if $P\in\inscrit_\alpha(R)$;
\item $\inscrit_\alpha(R)$ is the image of $\mmax_{T_\alpha}$ under the canonical restriction map $\mmax_{T_\alpha}\longrightarrow\mmax_R$;
\item $\inscrit_\alpha(R)$ is closed in $\mmax_R$.
\end{enumerate}
\end{lemma}
\begin{proof}
We proceed by induction. If $\alpha=0$, then $\inscrit_0(R)=\mmax$ and the claim is obvious; if $\alpha=1$, then $\inscrit_1(R)=\inscrit(R)$ is closed by Proposition \ref{prop:crit-closed}; thus $\inscrit(R)$ is the image of $\mmax_{T_1}$ by Lemma \ref{lemma:closed-intersect}.

Suppose that the three claims hold for every $\lambda<\alpha$. If $\alpha$ is a limit ordinal, then $\inscrit_\alpha(R)=\bigcap_{\lambda<\alpha}\inscrit_\lambda(R)$ is closed, and thus also $\inscrit_\alpha(R)$ is closed; the other claims follow from Lemma \ref{lemma:closed-intersect} and the definition of $T_\alpha$. If $\alpha=\gamma+1$ is a successor ordinal, then by definition $P\in\inscrit_\alpha(R)$ if and only if $PT_\gamma\in\inscrit(T_\gamma)$; therefore, the restriction map $\mmax_{T_\gamma}\longrightarrow\mmax_R$ sends $\inscrit(T_\gamma)$ to $\inscrit_{\gamma+1}(R)=\inscrit_\alpha(R)$. By the case $\alpha=1$, $\inscrit(T_\gamma)$ is closed in $\mmax_{T_\gamma}$; since the restriction map is closed, it follows that $\inscrit_\alpha(R)$ is closed in $\mmax_R$. The other two claims follow again from Lemma \ref{lemma:closed-intersect}. 
\end{proof}

\begin{lemma}\label{lemma:inscrit-deriv}
Let $R$ be an almost Dedekind domain. Then, $\inscrit_\alpha(R)\subseteq\deriv^\alpha(\mmax)$ for every ordinal $\alpha$.
\end{lemma}
\begin{proof}
By induction on $\alpha$. If $\alpha=1$ the claim is exactly Proposition \ref{prop:crit-deriv}. If $\alpha=\gamma+1$ is a successor ordinal, by Lemma \ref{lemma:incritalpha} the restriction map $\theta:\mmax_{T_\gamma}\longrightarrow\mmax_R$ establishes a homeomorphism between $\mmax_{T_\gamma}$ and $\inscrit_\gamma(R)$; therefore, by induction,
\begin{equation*}
\begin{aligned}
\inscrit_\alpha(R)& =\theta(\inscrit(T_\gamma))\subseteq\theta(\deriv(\mmax_{T_\gamma}))=\\
&=\deriv(\inscrit_\gamma(R))\subseteq\deriv(\deriv^\gamma(\mmax))=\deriv^\alpha(R).
\end{aligned}
\end{equation*}

If $\alpha$ is a limit ordinal then, by induction,
\begin{equation*}
\inscrit_\alpha(R)=\bigcap_{\beta<\alpha}\inscrit_\beta(R)\subseteq\bigcap_{\beta<\alpha}\deriv^\beta(\mmax)=\deriv^\alpha(\mmax).
\end{equation*}
The claim is proved.
\end{proof}

\begin{prop}\label{prop:scat-SPscat}
Let $R$ be an almost Dedekind domain. Then, the following hold.
\begin{enumerate}[(a)]
\item\label{prop:scat-SPscat:count} If $\mmax$ is countable, then it is scattered.
\item\label{prop:scat-SPscat:scat} If $\mmax$ is scattered, then $R$ is SP-scattered.
\end{enumerate}
\end{prop}
\begin{proof}
\ref{prop:scat-SPscat:count} If $\mmax$ is countable, then so is $\Spec(R)$. As a compact Hausdorff countable space, $\Spec(R)^\cons$ is scattered \cite{mazur-sierp-numerabili}, and a subspace of a scattered space is scattered; hence $\mmax$ is scattered.

\ref{prop:scat-SPscat:scat} Since $\mmax$ is scattered, there is an ordinal $\alpha$ such that $\deriv^\alpha(\mmax)$ is empty. By Lemma \ref{lemma:inscrit-deriv}, $\inscrit_\alpha(R)\subseteq\deriv^\alpha(\mmax)$, and thus $\inscrit_\alpha(R)$ is empty; in particular, $R$ is SP-scattered.
\end{proof}

When $R$ has finite rank, applying finitely many times the reasoning after Proposition \ref{prop:sgmi-inv} we will get the zero module, and thus an isomorphism $\Inv(T_n)\simeq\funcontcomp(\mmax_{T_n},\insZ)=\funcontcomp(\inscrit_n(R),\insZ)$; using its freeness, we can pull back this isomorphism to get a decomposition
\begin{equation*}
\Inv(R)\simeq\bigoplus_{i=0}^n\funcontcomp(\inscrit_i(R)\setminus\inscrit_{i+1}(R),\insZ).
\end{equation*}
We now want to prove this result to arbitrary SP-scattered domain, and for this we need two group-theoretic lemmas.
\begin{lemma}\label{lemma:unionesgr}
Let $G$ be an abelian group. Let $\{G_\lambda\}_{\lambda\in\Lambda}$ be a well-ordered set of ascending chain of subgroups of $G$ such that:
\begin{itemize}
\item each $G_\lambda$ is free;
\item for all $\lambda$, $G_{\lambda+1}\simeq G_\lambda\oplus H_\lambda$ for some subgroup $H_\lambda$;
\item if $\lambda$ is a limit ordinal, then $G_\lambda=\bigcup_{\alpha<\lambda}G_\alpha$;
\item $G=\bigcup_\lambda G_\lambda$.
\end{itemize}
Then, $\displaystyle{G\simeq G_0\oplus\bigoplus_{\lambda\in\Lambda} H_\lambda}$, and in particular it is free.
\end{lemma}
\begin{proof}
The statement is equivalent to \cite[Chapter 3, Lemma 7.3]{fuchs-abeliangroups}.
\end{proof}

\begin{lemma}\label{lemma:exseq-kernel}
Let $G_1,G_2,G_3$, and let $\phi_1:G_1\longrightarrow G_2$ and $\phi_2:G_2\longrightarrow G_3$ be surjective maps. Then, there is an exact sequence
\begin{equation*}
0\longrightarrow\ker\phi_1\longrightarrow\ker(\phi_2\circ\phi_1)\xrightarrow{~{\phi_1}~} \ker\phi_2\longrightarrow 0
\end{equation*}
\end{lemma}
\begin{proof}
Immediate.
\end{proof}

\begin{prop}\label{prop:Kalpha}
Let $R$ be an almost Dedekind domain, and let $\alpha$ be an ordinal. Then, there is an exact sequence
\begin{equation*}
0\longrightarrow\bigoplus_{i<\alpha}\funcontcomp(\inscrit_i(R)\setminus\inscrit_{i+1}(R),\insZ)\longrightarrow \Inv(R)\longrightarrow\Inv(T_\alpha)\longrightarrow 0.
\end{equation*}
\end{prop}
\begin{proof}
Since $R$ is Pr\"ufer, the extension map $I\mapsto IT$ is a surjective homomorphism from $\Inv(R)$ to $\Inv(T_\alpha)$, with kernel $K_\alpha:=\{I\in\Inv(R)\mid IT=T\}=\{I\in\Inv(R)\mid \supp(\nu_I)\cap\inscrit_\alpha(R)=\emptyset\}$. We need to show that $K_\alpha\simeq\bigoplus_{i<\alpha}\funcontcomp(\inscrit_i(R)\setminus\inscrit_{i+1}(R),\insZ)$, and we do so by induction on $\alpha$.

If $\alpha=0$ there is nothing to prove. If $\alpha=1$, the claim is exactly Proposition \ref{prop:sigcap-iso}. Suppose now that the claim holds for all $\lambda<\alpha$.

If $\alpha=\gamma+1$ is a successor ordinal, then the surjective map $\Inv(R)\longrightarrow\Inv(T_\alpha)$ factors through $\Inv(T_\gamma)$; by Lemma \ref{lemma:exseq-kernel}, we have an exact sequence
\begin{equation}\label{eq:Kseq}
0\longrightarrow K_\gamma\longrightarrow K_\alpha\longrightarrow K_{\gamma,\alpha}\longrightarrow 0,
\end{equation}
where $K_{\gamma,\alpha}$ is the kernel of $\Inv(T_\gamma)\longrightarrow\Inv(T_\alpha)$. By Proposition \ref{prop:sigcap-iso}, 
\begin{equation*}
K_{\gamma,\alpha}\simeq\funcontcomp(\mmax_{T_\gamma}\setminus\inscrit(T_\gamma),\insZ)\simeq \funcontcomp(\inscrit_\gamma(R)\setminus\inscrit_\alpha(R),\insZ),
\end{equation*}
which is free by \cite[Satz 1]{nobeling}; hence, \eqref{eq:Kseq} splits as $K_\alpha\simeq K_{\gamma,\alpha}\oplus K_\gamma$, and the claim now follows by induction.

Suppose now that $\alpha$ is a limit ordinal. Consider the sequence $\{K_\lambda\}_{\lambda<\alpha}$: by induction, it is an ascending chain of free subgroups of $K_\alpha$, and it satisfies the two middle conditions of Lemma \ref{lemma:unionesgr} with $H_\lambda=\funcontcomp(\inscrit_\lambda(R)\setminus\inscrit_{\lambda+1}(R),\insZ)$. We show that $K_\alpha=\bigcup_{\lambda<\alpha}K_\lambda$. 

Indeed, suppose $I\in K_\alpha$: then,
\begin{equation*}
\supp\nu_I\subseteq\mmax\setminus\inscrit_\alpha(R)=\bigcup_{\lambda<\alpha}(\mmax\setminus\inscrit_\lambda(R))
\end{equation*}
since $\alpha$ is a limit ordinal. Since each $\inscrit_\lambda(R)$ is closed in the inverse topology (Lemma \ref{lemma:incritalpha}), $\{\mmax\setminus\inscrit_\lambda(R)\}_{\lambda<\alpha}$ is an open cover of the compact set $\supp\nu_I$, and since $\{\mmax\setminus\inscrit_\lambda(R)\}_{\lambda<\alpha}$ is also a chain it means that there is a $\overline{\lambda}$ such that $\supp\nu_I\subseteq \mmax\setminus\inscrit_{\overline{\lambda}}(R)$. Hence, $I\in K_{\overline{\lambda}}$, and so $K_\alpha$ is the union of the $K_\lambda$. From Lemma \ref{lemma:unionesgr} we now get that $K_\alpha$ is free and that it has the claimed decomposition.
\end{proof}

\begin{teor}\label{teor:SP-scattered}
Let $R$ be an SP-scattered almost Dedekind domain with SP-rank $\alpha$. Then,
\begin{equation*}
\Inv(R)\simeq\bigoplus_{i<\alpha}\funcontcomp(\inscrit_i(R)\setminus\inscrit_{i+1}(R),\insZ);
\end{equation*}
in particular, $\Inv(R)$ is free.
\end{teor}
\begin{proof}
By definition, $T_\alpha=K$ and so $\Inv(T_\alpha)=(0)$. Hence, $\Inv(R)$ is homeomorphic to $K_\alpha$, the kernel of the extension map $\Inv(R)\longrightarrow\Inv(T_\alpha)$. The claim now follows from Proposition \ref{prop:Kalpha}.
\end{proof}

\begin{oss}
The isomorphism in Theorem \ref{teor:SP-scattered} is not canonical. Indeed, suppose that $R$ has a single critical maximal ideal $M$, so that $T_1=R_M$ and $T_2=K$. Let $X:=\mmax\setminus\inscrit(R)$. Then, by the theorem,
\begin{equation*}
\Inv(R)\simeq\funcontcomp(X,\insZ)\oplus\funcontcomp(\inscrit(R),\insZ)= \funcontcomp(X,\insZ)\oplus \nu_I\insZ,
\end{equation*}
where $I$ is an invertible ideal such that $\nu_I$ is not continuous. However, there is no way to choose $I$ canonically; indeed, if $J$ is an invertible ideal with $\nu_J$ continuous, then we also have
\begin{equation*}
\Inv(R)\simeq\funcontcomp(X,\insZ)\oplus\funcontcomp(\inscrit(R),\insZ)= \funcontcomp(X,\insZ)\oplus \nu_{IJ}\insZ,
\end{equation*}
and the first component of a $\nu_L$ will be different in the two decompositions. On the other hand, by Proposition \ref{prop:sigcap-iso}, the \emph{second} component will always be the same, since it coincides with $\nu_L|_{\inscrit(R)}$, i.e., in this case, to $\nu_L(M)$.
\end{oss}

\begin{oss}
If $R$ is not SP-scattered, then the exact sequence of Proposition \ref{prop:Kalpha} is still valid when $\alpha$ is the SP-rank of $R$; in particular, if $\Inv(T_\alpha)$ is free then the sequence is split and $\Inv(R)$ is free. Therefore, to prove that $\Inv(R)$ is free for every almost Dedekind domain $R$, it is enough to prove it for the domains $R$ such that $\inscrit(R)=\Max(R)$.
\end{oss}

\section{Divisorial ideals}\label{sect:div}
Let $R$ be a ring. A fractional ideal $I$ of $R$ is said to be \emph{divisorial} if $(R:(R:I))=I$ or, equivalently, if $I$ is an intersection of principal ideals; the set $\Div(R)$ of divisorial ideals contains $\Inv(R)$ and is a monoid under the ``$v$-product'' $I\ast_v J:=(R:(R:IJ))$ \cite[Section 34]{gilmer}. When $R$ is completely integrally closed, $\Div(R)$ is a group \cite[Theorem 34.3]{gilmer}; for example, this happens when $R$ is a one-dimensional Pr\"ufer domain, and in particular when $R$ is an almost Dedekind domain.

In this context, the relationship between $\Div(R)$ and $\Inv(R)$ is better understood in the context of lattice-ordered groups \cite[Section 3]{HK-Olb-Re}. A \emph{lattice-ordered group} (or \emph{$\ell$-group} for short) is a group $(G,+)$ endowed with a partial order $\leq$ such that, for every $x,y\in G$, the infimum $x\wedge y$ and the supremum $x\vee y$ exists, and such that whenever $x,y,g\in G$ and $x\leq y$ then $x+g\leq y+g$. An \emph{$\ell$-homomorphism} of $\ell$-groups is a homomorphism $\phi:G\longrightarrow G'$ such that $\phi(x\wedge y)=\phi(x)\wedge\phi(y)$ for every $x,y\in G$ (equivalently, such that $\phi(x\vee y)=\phi(x)\vee\phi(y)$ for every $x,y\in G$). 

An $\ell$-group $G$ is said to be \emph{complete} if, whenever a set of elements of $G$ is bounded below, it has an infimum; the \emph{completion} of an $\ell$-group $G$ is the smallest complete $\ell$-group $H$ containing $G$ (and such that the inclusion is an $\ell$-homomorphism). Every $\ell$-group has a completion; moreover, if $G$ is an $\ell$-group and $H$ is an $\ell$-group containing $G$, then $H$ is the completion of $G$ if and only if $H$ is complete and, for each $0<h\in H$, there are $g_1,g_2\in H$ such that $0<g_1\leq h\leq g_2$ \cite[Theorem 2.4]{completion-lgroups}. 

The group $\Inv(R)$ of invertible ideals of a Pr\"ufer domain $R$ (so, in particular, of an almost Dedekind domain $R$) is an $\ell$-group if we set that $I\leq J$ if and only if $I\supseteq J$; the infimum $I\wedge J$ is the sum $I+J$, while the supremum $I\vee J$ is the intersection $I\cap J$. If $T$ is an overring of $R$, then the canonical map $\phi:\Inv(R)\longrightarrow\Inv(T)$ is an $\ell$-homomorphism, since $\phi(I+J)=(I+J)T=IT+JT=\phi(I)+\phi(J)$. When $R$ is also one-dimensional, its completion is exactly the group $\Div(R)$ of divisorial ideals of $R$ \cite[Proposition 3.1]{HK-Olb-Re}. 

The group $\inscont(X,\insZ)$ of continuous functions is an $\ell$-group, where the order is the componentwise order; its completion can be expressed in terms of the \emph{Gleason cover} $E_X$ of $X$. A topological space is said to be \emph{extremally disconnected} if the closure of every open set is open. It can be shown that, if $X$ is a regular Hausdorff space, then there is a ``minimal'' extremally disconnected space $E_X$ endowed with a map $j:E_X\longrightarrow X$ that is \emph{perfect} (i.e., closed, continuous and compact) and surjective, and that extremally disconnected spaces are the projective objects in the category of regular Hausdorff spaces and perfect maps. Such a space $E_X$ can also be constructed explicitly as the space of convergent ultrafilters on the open sets of $X$ \cite{strauss-extremallydisconnected}. 

The following two results extend \cite[Lemma 5.2 and Theorem 5.3]{HK-Olb-Re} to non-compact spaces; the proof is essentially the same.
\begin{lemma}\label{lemma:completion-funcontcomp}
Let $X$ be a Hausdorff regular space. Then, the natural embedding $\psi:\funcontcomp(X,\insZ)\longrightarrow\funcontcomp(E_X,\insZ)$, $f\mapsto f\circ j$, makes $\funcontcomp(E_X,\insZ)$ into the completion of $\psi(\funcontcomp(X,\insZ))$ as an $\ell$-group. Moreover, the quotient group $\funcontcomp(E_X,\insZ)/\psi(\funcontcomp(X,\insZ))$ is free.
\end{lemma}
\begin{proof}
We first show that $\psi$ is well-defined. Clearly $f\circ j$ is continuous; since $j$ is compact, by \cite[Lemma 4]{strauss-extremallydisconnected} $j^{-1}(Z)$ is compact for every compact subspace $Z$ of $X$; hence, $\supp(j\circ f)\subseteq j^{-1}(\supp(f))$ is the closed subset of a compact space, and thus it is compact. Hence $f\circ j$ has compact support and $\psi$ is well-defined.

We now claim that $\funcontcomp(E_X,\insZ)$ and $\psi(\funcontcomp(X,\insZ))$ are Specker groups of $\insfunct(E_X,\insZ)$. For every $f\in\funcontcomp(E_X,\insZ)$ and every $n\inN\setminus\{0\}$, the set $f^{-1}(n)\subseteq\supp(f)$ is a clopen subset of $E_X$, and thus its characteristic function is continuous and has compact support. Hence $\chi_{f^{-1}(n)}\in\funcontcomp(E_X,\insZ)$; by definition, $\funcontcomp(E_X,\insZ)$ is Specker.

Likewise, if $g\in\funcontcomp(X,\insZ)$, then $\chi_{g^{-1}(n)}\in\funcontcomp(X,\insZ)$. Since $\psi(g)^{-1}(n)=j^{-1}(g^{-1}(n))$, we have $\chi_{\psi(g)^{-1}(n)}=\psi(\chi_{g^{-1}(n)})\in\psi(\funcontcomp(X,\insZ))$, and thus also $\psi(\funcontcomp(X,\insZ))$ is Specker. 

By \cite[Satz 2]{nobeling} there is a free subgroup $S$ such that $\funcontcomp(E_X,\insZ)=\psi(\funcontcomp(X,\insZ))\oplus S$. Hence, the quotient $\funcontcomp(E_X,\insZ)/\psi(\funcontcomp(X,\insZ))$ is isomorphic to $S$ and thus free.

\medskip

It only remains to prove that $\funcontcomp(E_X,\insZ)$ is the completion of $\psi(\funcontcomp(X,\insZ))$.

By \cite[Proposition 3.29]{mcgovern-rigid}, the space $\mathcal{C}(E_X,\insZ)$ is a complete $\ell$-group; if now $F:=\{f_\alpha\}$ is a subset of $\funcontcomp(E_X,\insZ)$ that is bounded below, say $0<g\leq f_\alpha$ for every $\alpha$, then $F$ is also bounded below in $\mathcal{C}(E_X,\insZ)$, and thus $F$ has an infimum $\overline{f}$ in $\mathcal{C}(E_X,\insZ)$. However, $\supp\overline{g}\subseteq\supp\overline{f_\alpha}$ for every $\alpha$, and thus $\supp\overline{g}$ is a closed set of a compact set, and in particular it is compact. Hence $\overline{g}\in\funcontcomp(E_X,\insZ)$ and $\funcontcomp(E_X,\insZ)$ is complete.

To prove that $\funcontcomp(E_X,\insZ)$ is the completion of $\funcontcomp(X,\insZ)$, we need to prove that if $0<f\in\funcontcomp(X,\insZ)$ then there are $g_1,g_2\in\funcontcomp(E_X,\insZ)$ such that $0<\psi(g_1)\leq f\leq\psi(g_2)$. The support of $f$ is clopen, and thus $j^{-1}(\supp(f))$ is an open set of $E_X$; therefore, $j^{-1}(\supp(f))$ is extremally disconnected and, in fact, $j:j^{-1}(\supp(f))\longrightarrow\supp(f)$ is the minimal perfect mapping onto $\supp(f)$, so that $j^{-1}(\supp(f))\simeq E_{\supp(f)}$. By \cite[Lemma 5.2]{HK-Olb-Re}, we can find $\widetilde{g}_1,\widetilde{g}_2\in\mathcal{C}(j^{-1}(\supp(f)),\insZ)$ with the right properties; since $\widetilde{g}_1,\widetilde{g}_2$ can be extended to continuous functions of compact support $g_1,g_2:E_X\longrightarrow\insZ$, we get $0<\psi(g_1)\leq f\leq\psi(g_2)$, as claimed. The claim is proved.
\end{proof}

\begin{teor}\label{teor:Div}
Let $R$ be an SP-domain. Then, there is a commutative diagram
\begin{equation*}
\begin{CD}
\Inv(R)  @>{\subseteq}>> \Div(R)\\
@VV{\Psi}V    @V{\beta}VV\\
\funcontcomp(\mmax,\insZ) @>{\psi}>> \funcontcomp(E_\mmax,\insZ)
\end{CD}
\end{equation*}
such that the vertical arrows are isomorphisms. In particular, $\Div(R)\simeq\funcontcomp(E_X,\insZ)$ and the quotient $\Div(R)/\Inv(R)$ are free groups.
\end{teor}
\begin{proof}
By \cite[Proposition 3.1]{HK-Olb-Re}, $\Div(R)$ is the completion (as an $\ell$-group) of $\Inv(R)$; on the other hand, by Lemma \ref{lemma:completion-funcontcomp}, $\funcontcomp(E_\mmax,\insZ)$ is the completion of $\psi(\funcontcomp(\mmax,\insZ))$. Therefore, $\beta$ must be an isomorphism.

In particular, $\Div(R)$ is isomorphic to $\funcontcomp(E_\mmax,\insZ)$, and thus free by \cite[Satz 1]{nobeling}. Moreover, $\Div(R)/\Inv(R)$ is isomorphic to the quotient $\funcontcomp(E_\mmax,\insZ)/\psi(\funcontcomp(\mmax,\insZ))$; by Lemma \ref{lemma:completion-funcontcomp}, this group is free too.
\end{proof}

It is temping to apply the same methods used in the previous theorem for the case of SP-scattered domains. However, while the map $\Inv(R)\longrightarrow\Inv(T_\alpha)$ is a $\ell$-homomorphism for every $\alpha$, the decomposition $\Inv(R)\simeq\bigoplus_{i<\alpha}\funcontcomp(X_i,\insZ)$ of Theorem \ref{teor:SP-scattered} (where $X_i:=\inscrit_i(R)\setminus\inscrit_{i+1}(R)$) is \emph{not} a decomposition of $\ell$-groups, in the sense that the isomorphism map is not an $\ell$-homomorphism, since the order on $\Inv(R)$ does not correspond to the componentwise order of the direct sum.

For example, suppose that $\mmax=\{M_0,M_1,\ldots\}\cup\{M_\infty\}=X\cup\{M_\infty\}$, where $M_i$ is isolated for $i\inN$, while $M_\infty$ is critical (and thus a non-isolated point of $\mmax$). By Theorem \ref{teor:SP-scattered}, we have
\begin{equation*}
\Inv(R)\simeq\funcontcomp(X,\insZ)\oplus\funcontcomp(\{M_\infty\},\insZ);
\end{equation*}
in particular, the first summand is isomorphic to the direct sum of countably many copies of $\insZ$, while the second summand is isomorphic to $\insZ$.

The generator of the second summand in $\Inv(R)$ is an invertible ideal $I$ that is contained in $M_\infty$ and is not continuous; in particular, there must be a maximal ideal $M_i$ containing $I$, so that $I\geq M_i$ in $\Inv(R)$. However, in the direct sum representation $I$ corresponds to $(\mathbf{0},n)$, where $\mathbf{0}$ is the zero function on $X$ and $n>0$, while $P$ correspond to $(\chi_{\{P\}},0)$: in particular, $(\mathbf{0},1)$ and $(\chi_{\{P\}},0)$ are not comparable in the componentwise order of $\funcontcomp(X,\insZ)\oplus\funcontcomp(\{M_\infty\},\insZ)$.

Note also that, in this case, both $\funcontcomp(X,\insZ)$ and $\funcontcomp(\{M_\infty\},\insZ)$ are complete $\ell$-groups, and thus by \cite[Corollary to Theorem 2.4]{completion-lgroups} so is their direct sum (with the componentwise order); however, if $J$ is a radical ideal contained in $M_\infty$, then $J$ is divisorial (since it is the intersection of the $M_i$ containing it) but not invertible (otherwise $M_\infty$ would not be critical). Hence, we don't have a natural decomposition $\Div(R)\simeq\funcontcomp(E_X,\insZ)\oplus\funcontcomp(E_{\{M_\infty\}},\insZ)$, and thus the analogue of Theorem \ref{teor:SP-scattered}, i.e., the decomposition
\begin{equation*}
\Div(R)\simeq\bigoplus_{i<\alpha}\funcontcomp(E_{X_i},\insZ)
\end{equation*}
does \emph{not} hold, in general.

\section{Length functions}\label{sect:length}
Let $R$ be an integral domain and let $\Mod(R)$ be the set of $R$-modules. A \emph{singular length function} on $R$ is a map $\ell:\Mod(R)\longrightarrow\{0,\infty\}$ such that:
\begin{itemize}
\item $\ell(0)=0$;
\item for every short exact sequence $0\longrightarrow M_1\longrightarrow M_2\longrightarrow M_3\longrightarrow 0$, we have $\ell(M_2)=\ell(M_1)+\ell(M_3)$;
\item for every module $M$, we have $\ell(M)=\sup\{\ell(N)\mid N$ is a finitely generated submodule of $M\}$.
\end{itemize}
Every length function $\ell$ is uniquely determined by the map $\tau$ associating to each ideal $I$ of $R$ the length $\ell(R/I)$ \cite[Proposition 3.3]{zanardo_length}. We call $\tau$ the \emph{ideal colength associated to $\ell$}. Given two ideals $I\subseteq J$, we have $\tau(I)\geq\tau(J)$. See \cite{northcott_length} and \cite{vamos-additive} for an introduction to length functions.

If $\ell$ is a length function on $R$ and $T$ is a flat overring of $R$, we define $\ell\otimes T$ as the length function on $R$ such that $(\ell\otimes T)(M)=\ell(M\otimes T)$ for all modules $M$ \cite[Section 3]{length-funct}; likewise, if $\tau$ is the corresponding ideal colength, we set $(\tau\otimes T)(I):=(\ell\otimes T)(R/I)=\ell(T/IT)$, i.e., $\tau\otimes T$ is the colength corresponding to $\ell\otimes T$. The functions $\ell\otimes T$ and $\tau\otimes T$ can also be seen as length function on $T$.

Let $\inssubmod(R)$ be the set of $R$-submodules of $K$. A \emph{stable semistar operation} on $R$ is a closure operation $\star:\inssubmod(R)\longrightarrow\inssubmod(R)$, $I\mapsto I^\star$, such that:
\begin{itemize}
\item $x\cdot I^\star=(xI)^\star$ for every $x\in K$, $I\in\inssubmod(R)$;
\item $(I\cap J)^\star=I^\star\cap J^\star$ for every $I,J\in\inssubmod(R)$.
\end{itemize}

There is a natural bijective correspondence between singular length functions and stable semistar operations; a length function $\ell$ with ideal colength $\tau$ and a semistar operation $\star$ are associated when $\tau(I)=0$ if and only if $1\in I^\star$ \cite[Section 6]{length-funct}.

Let $R$ be an almost Dedekind domain. If $\mmax$ is scattered, then every stable semistar operation $\star$ on $R$ is in the form $I\mapsto\bigcap\{IR_P\mid P\in\Delta\}$ for some $\Delta\subseteq\Max(R)$ \cite[Corollaries 8.6 and 8.9]{jaff-derived}; this writing corresponds to a decomposition $\ell=\sum\{\ell\otimes R_P\mid P\in\Delta\}$ (see \cite[Section 3]{length-funct} and \cite[Proposition 7.7]{jaff-derived}). In particular, $1\in I^\star$ if and only if $V(I)\cap\Delta$ is empty; hence, $1\in I^\star$ if and only if $1\in\rad(I)^\star$, and
\begin{equation*}
\tau(I)=\tau(\rad(I))
\end{equation*}
for every ideal $I$. In this section, we want to prove that this equality holds, more generally, for all SP-scattered domains.

\begin{lemma}\label{lemma:sum}
Let $\ell$ be a singular length function on $R$, and let $\tau$ be the corresponding ideal colength. Then, for every $I,J$, we have
\begin{equation*}
\tau(IJ)=\tau(I)+\tau(J).
\end{equation*}
\end{lemma}
\begin{proof}
If $\tau(IJ)=0$, then $\tau(I)=0=\tau(J)$ since $IJ$ is contained in both $I$ and $J$. Conversely, if $\tau(I)=\tau(J)=0$, let $\star$ be the stable semistar operation associated to $\ell$: then, $I^\star=R^\star=J^\star$, and thus
\begin{equation*}
(IJ)^\star=(I^\star J^\star)^\star=(R^\star R^\star)^\star=R^\star;
\end{equation*}
hence $\tau(IJ)=0$. Thus $\tau(IJ)=0$ if and only if $\tau(I)=\tau(J)=0$, and since $\ell$ is singular it follows that $\tau(IJ)=\infty$ if and only if at least one of $\tau(I)$ and $\tau(J)$ is equal to $\infty$. The claim follows.
\end{proof}

\begin{lemma}\label{lemma:potpan}
Let $\ell$ be a singular length function on $R$, and let $\tau$ be the corresponding ideal colength. Let $I,J$ be two ideals. If $I^n\subseteq J\subseteq I$ for some integer $n$, then $\tau(I)=\tau(J)$.
\end{lemma}
\begin{proof}
Since $I^n\subseteq J\subseteq I$, we have $\tau(I^n)\geq\tau(J)\geq\tau(I)$. By Lemma \ref{lemma:sum}, $\tau(I^n)=n\tau(I)$; since $\tau(I)$ is either $0$ or $\infty$, we have $n\tau(I)=\tau(I)$. Thus $\tau(I)=\tau(J)$, as claimed.
\end{proof}

\begin{lemma}\label{lemma:VIOMega}
Let $R$ be a one-dimensional Pr\"ufer domain, let $\Omega\subseteq\mmax$ be a closed set, let $T:=\bigcap\{R_P\mid P\in\Omega\}$ and let $I$ be an ideal such that $V(I)\subseteq\Omega$. Let $\ell$ be a singular length function on $R$, and let $\tau$ be the corresponding ideal colength. Then, $\tau(I)=(\tau\otimes T)(I)$.
\end{lemma}
\begin{proof}
Since $V(I)\subseteq\Omega$, we have $I=IT\cap R$. Let $\star$ be the spectral semistar operation associated to $\ell$: then, $I^\star=(IT)^\star\cap R^\star$, and thus $1\in I^\star$ if and only if $1\in(IT)^\star$. Therefore, $\tau(I)=0$ if and only if $(\tau\otimes T)(I)=0$. The claim follows.
\end{proof}

\begin{teor}\label{teor:length}
Let $R$ be an almost Dedekind domain, and let $\alpha$ be an ordinal. Let $\ell$ be a singular length function on $R$, and let $\tau$ be the corresponding ideal colength. Then, for every ideal $I$ we have
\begin{equation}\label{eq:sommatau}
\tau(I)=\tau(\rad(I))+(\tau\otimes T_\alpha)(I).
\end{equation}
\end{teor}
\begin{proof}
Suppose first that $\tau(I)=0$. Then, also $\tau(\rad(I))=0$; furthermore, $\ell\otimes T_\alpha\leq\ell$ (i.e., $(\ell\otimes T_\alpha)(M)\leq\ell(M)$ for every module $M$), and thus $\tau\otimes T_\alpha\leq\tau$. Hence also $(\tau\otimes T_\alpha)(I)=0$, and the claim follows.

Suppose now that $\tau(I)=\infty$, and let $\star$ be the spectral semistar operation associated to $\ell$: then, $1\notin I^\star$, and thus $J:=I^\star\cap R$ is a proper ideal of $R$ such that $J=J^\star\cap R$. In particular, $\tau(J)=\infty$. If $V(J)\subseteq\inscrit_\alpha(R)$, then by Lemma \ref{lemma:VIOMega} we have $\tau(J)=(\tau\otimes T_\alpha)(J)=\infty$. In particular, since $I\subseteq J$, we also have $(\tau\otimes T_\alpha)(I)=\infty$; thus, \eqref{eq:sommatau} holds.

Suppose now that $V(J)\nsubseteq\inscrit_\alpha(R)$, and let $\beta\leq\alpha$ be the minimal ordinal such that $V(J)$ is not contained in $\inscrit_\beta(R)$. Such a $\beta$ cannot be a limit ordinal (since in this case $\inscrit_\beta(R)$ is the intersection of $\inscrit_\gamma(R)$ for $\gamma<\beta$): therefore, $\beta=\gamma+1$ for some $\gamma$.  By construction, $V(J)\subseteq\inscrit_\gamma(R)$.

By Lemma \ref{lemma:VIOMega}, we have $\tau(J)=(\tau\otimes T_\gamma)(J)$, and since $\rad(JT_\gamma)=\rad(J)T_\gamma$, we also have $\tau(\rad(JT_\gamma))=\tau(\rad(J)T_\gamma)$; moreover, $(\tau\otimes T_\gamma)\otimes T_\alpha=\tau\otimes(T_\gamma\otimes T_\alpha)=\tau\otimes T_\alpha$. Since the same holds for $I$, we can substitute $R$ with $T_\gamma$, $I$ with $IT_\gamma$ and $J$ with $JT_\gamma$; equivalently, we can suppose without loss of generality that $\gamma=0$ and $\beta=1$.

Under this condition, there is a non-critical maximal ideal $P$ of $R$ such that $J\subseteq P$. Since $\inscrit(R)$ is a closed subset of $\mmax$, there is a clopen subset $\Omega$ such that $P\in\Omega$ and $\Omega\cap\inscrit(R)=\emptyset$. Let $S:=\bigcap\{R_P\mid P\in\Omega\}$ and $S^\perp:=\bigcap\{R_Q\mid Q\in\mmax\setminus\Omega\}$: then, since $\Omega$ is clopen, $\{S,S^\perp\}$ is a Jaffard family (see \cite[Section 6.3]{fontana_factoring}), so that $\tau(J)=(\tau\otimes S)(JS)+(\tau\otimes S^\perp)(JS^\perp)$ and $J^\star=(JS)^\star\cap(JS^\perp)^\star$ \cite[Theorem 3.10]{length-funct}. As $J=J^\star\cap R\subseteq P$, moreover, $(JS)^\star\subseteq PS\subsetneq S$ and thus $1\notin(JS)^\star$; hence $(\tau\otimes S)(JS)=\infty$.

Since $\Omega$ is disjoint from $\inscrit(R)$, $S$ is an SP-domain; in particular, every ideal $I$ contains a power of its radical, and thus by Lemma \ref{lemma:potpan} $(\tau\otimes S)(JS)=(\tau\otimes S)(\rad(JS))$. Therefore, $(\tau\otimes S)(\rad(JS))=\infty$; it follows that $\tau(\rad(J))=\infty$, and since $I\subseteq J$ also $\tau(\rad(I))=\infty$. Thus $\tau(I)=\tau(\rad(I))=\tau(\rad(I))+(\tau\otimes T_\alpha)(I)$, as claimed.
\end{proof}

\begin{cor}
Let $R$ be an SP-scattered almost Dedekind domain; let $\ell$ be a singular length function and let $\tau$ be the corresponding ideal colength. Then, for every ideal $I$, we have
\begin{equation*}
\tau(I)=\tau(\rad(I)).
\end{equation*}
\end{cor}
\begin{proof}
Since $R$ is SP-scattered, we have $T_\alpha=K$ for some ordinal $\alpha$. The claim now follows from Theorem \ref{teor:length} and the fact that $(\tau\otimes K)(I)=0$ for every nonzero ideal $I$.
\end{proof}

\bibliographystyle{plain}
\bibliography{/bib/articoli,/bib/libri,/bib/miei}

\begin{thebibliography}{10}

\bibitem{intD}
Paul-Jean Cahen and Jean-Luc Chabert.
\newblock {\em Integer-{V}alued {P}olynomials}, volume~48 of {\em Mathematical
  Surveys and Monographs}.
\newblock American Mathematical Society, Providence, RI, 1997.

\bibitem{completion-lgroups}
Paul Conrad and Donald McAlister.
\newblock The completion of a lattice ordered group.
\newblock {\em J. Austral. Math. Soc.}, 9:182--208, 1969.

\bibitem{spectralspaces-libro}
Max Dickmann, Niels Schwartz, and Marcus Tressl.
\newblock {\em Spectral spaces}, volume~35 of {\em New Mathematical
  Monographs}.
\newblock Cambridge University Press, Cambridge, 2019.

\bibitem{compact-intersections}
Carmelo~A. Finocchiaro and Dario Spirito.
\newblock Topology, intersections and flat modules.
\newblock {\em Proc. Amer. Math. Soc.}, 144(10):4125--4133, 2016.

\bibitem{fontana_factoring}
Marco Fontana, Evan Houston, and Thomas Lucas.
\newblock {\em Factoring {I}deals in {I}ntegral {D}omains}, volume~14 of {\em
  Lecture Notes of the Unione Matematica Italiana}.
\newblock Springer, Heidelberg; UMI, Bologna, 2013.

\bibitem{fuchs-abeliangroups}
L\'{a}szl\'{o} Fuchs.
\newblock {\em Abelian groups}.
\newblock Springer Monographs in Mathematics. Springer, Cham, 2015.

\bibitem{gilmer}
Robert Gilmer.
\newblock {\em Multiplicative {I}deal {T}heory}.
\newblock Marcel Dekker Inc., New York, 1972.
\newblock Pure and Applied Mathematics, No. 12.

\bibitem{gilmer-overrings}
Robert~W. Gilmer, Jr.
\newblock Overrings of {P}r\"{u}fer domains.
\newblock {\em J. Algebra}, 4:331--340, 1966.

\bibitem{HK-Olb-Re}
Olivier~A. Heubo-Kwegna, Bruce Olberding, and Andreas Reinhart.
\newblock Group-theoretic and topological invariants of completely integrally
  closed {P}r\"{u}fer domains.
\newblock {\em J. Pure Appl. Algebra}, 220(12):3927--3947, 2016.

\bibitem{mcgovern-rigid}
Michelle~L. Knox and Warren~Wm. McGovern.
\newblock Rigid extensions of {$l$}-groups of continuous functions.
\newblock {\em Czechoslovak Math. J.}, 58(133)(4):993--1014, 2008.

\bibitem{mazur-sierp-numerabili}
Stefan Mazurkiewicz and Wac\l~aw Sierpi\'{n}ski.
\newblock Contribution `a la topologie des ensembles d\'enombrables.
\newblock {\em Fundam. Math.}, 1:17--27, 1920.

\bibitem{nakano-almded}
Noboru Nakano.
\newblock Idealtheorie in einem speziellen unendlichen algebraischen
  {Z}ahlk\"{o}rper.
\newblock {\em J. Sci. Hiroshima Univ. Ser. A}, 16:425--439, 1953.

\bibitem{nobeling}
G.~N\"{o}beling.
\newblock Verallgemeinerung eines {S}atzes von {H}errn {E}. {S}pecker.
\newblock {\em Invent. Math.}, 6:41--55, 1968.

\bibitem{northcott_length}
Douglas~Geoffrey Northcott and Manfred Reufel.
\newblock A generalization of the concept of length.
\newblock {\em Quart. J. Math. Oxford Ser. (2)}, 16:297--321, 1965.

\bibitem{olberding-factoring-SP}
Bruce Olberding.
\newblock Factorization into radical ideals.
\newblock In {\em Arithmetical properties of commutative rings and monoids},
  volume 241 of {\em Lect. Notes Pure Appl. Math.}, pages 363--377. Chapman \&
  Hall/CRC, Boca Raton, FL, 2005.

\bibitem{specker}
Ernst Specker.
\newblock Additive {G}ruppen von {F}olgen ganzer {Z}ahlen.
\newblock {\em Portugal. Math.}, 9:131--140, 1950.

\bibitem{starloc}
Dario Spirito.
\newblock Jaffard families and localizations of star operations.
\newblock {\em J. Commut. Algebra}, 11(2):265--300, 2019.

\bibitem{length-funct}
Dario Spirito.
\newblock Decomposition and classification of length functions.
\newblock {\em Forum Math.}, 32(5):1109--1129, 2020.

\bibitem{jaff-derived}
Dario Spirito.
\newblock The derived sequence of a pre-{J}affard sequence.
\newblock {\em Mediterr. J. Math.}, to appear.

\bibitem{strauss-extremallydisconnected}
Dona~Papert Strauss.
\newblock Extremally disconnected spaces.
\newblock {\em Proc. Amer. Math. Soc.}, 18:305--309, 1967.

\bibitem{vamos-additive}
Peter V\'amos.
\newblock Additive functions and duality over {N}oetherian rings.
\newblock {\em Quart. J. Math. Oxford Ser. (2)}, 19:43--55, 1968.

\bibitem{vaughan-SP}
N.~H. Vaughan and R.~W. Yeagy.
\newblock Factoring ideals into semiprime ideals.
\newblock {\em Canadian J. Math.}, 30(6):1313--1318, 1978.

\bibitem{zanardo_length}
Paolo Zanardo.
\newblock Multiplicative invariants and length functions over valuation
  domains.
\newblock {\em J. Commut. Algebra}, 3(4):561--587, 2011.

\end{thebibliography}
\end{document}